\documentclass[preprint,11pt]{elsarticle}

\usepackage{amsfonts, amsmath, amscd}
\usepackage[psamsfonts]{amssymb}

\usepackage{amssymb}

\usepackage{pb-diagram}

\usepackage[all,cmtip]{xy}

\usepackage[usenames]{color}

\newtheorem{theorem}{Theorem}[section]
\newtheorem{lemma}[theorem]{Lemma}
\newtheorem{corollary}[theorem]{Corollary}

\newtheorem{remark}[theorem]{Remark}
\newtheorem{proposition}[theorem]{Proposition}
\newtheorem{definition}[theorem]{Definition}
\newtheorem{example}[theorem]{Example}

\newproof{proof}{Proof}

\numberwithin{equation}{section}
\numberwithin{theorem}{section}
\newcommand{\e}{\varepsilon}
\newcommand{\w}{\omega}

\newcommand{\NN}{\mathbb{N}}

\newcommand{\IR}{\mathbb{R}}

\newcommand{\IF}{\mathbb{F}}

\newcommand{\xxx}{\mathbf{x}}

\newcommand{\TTT}{\mathcal{T}}
\newcommand{\FF}{\mathcal{F}}

\newcommand{\BB}{\mathcal{B}}

\newcommand{\Nn}{\mathcal{N}}
\newcommand{\AAA}{\mathcal{A}}

\newcommand{\LL}{\mathcal{L}}

\newcommand{\supp}{\mathrm{supp}}

\newcommand{\cl}{\mathrm{cl}}
\newcommand{\Ra}{\Rightarrow}

\newcommand{\cacx}{\overline{\mathrm{acx}}}

\newcommand{\Bo}{\mathsf{Bo}}

\newcommand{\Id}{\mathsf{id}}

\newcommand{\ind}{\mathsf{ind}}

\newcommand{\spn}{\mathsf{span}}
\newcommand{\cspn}{\overline{\mathsf{span}}}

\newcommand{\CC}{C_k}

\newcommand{\SI}{\underrightarrow{\mbox{ s-$\mathsf{ind}$}}_n\,}
\newcommand{\SM}{{\setminus}}

\begin{document}

\begin{frontmatter}

\title{Dunford--Pettis type  properties of locally convex spaces}

\author{S.~Gabriyelyan}%\tnotetext[label1]{The first named author was partially supported by Israel Science Foundation grant 1/12.}
\ead{saak@math.bgu.ac.il}
\address{Department of Mathematics, Ben-Gurion University of the Negev, Beer-Sheva, P.O. 653, Israel}

\begin{abstract}
In 1953, Grothendieck introduced and studied the  Dunford--Pettis property (the $DP$ property) and the strict  Dunford--Pettis property (the strict $DP$ property). The $DP$ property of order $p\in[1,\infty]$ for Banach spaces was introduced by Castillo and Sanchez in 1993. Being motivated by these notions, for $p,q\in[1,\infty]$,  we define the strict Dunford--Pettis property of order $p$ (the strict $DP_p$ property) and the  sequential Dunford--Pettis property of order $(p,q)$ (the sequential $DP_{(p,q)}$ property). We show that a locally convex space (lcs) $E$ has the $DP$ property iff the space $E$ endowed with the Grothendieck topology $\tau_{\Sigma'}$ has the weak Glicksberg property, and  $E$ has the strict $DP_p$ property iff the space $(E,\tau_{\Sigma'}) $ has the $p$-Schur property. We also characterize lcs with the sequential $DP_{(p,q)}$ property. Some permanent properties and relationships between Dunford--Pettis type properties are studied. Numerous (counter)examples are given. In particular, we give the first example of an lcs with the strict $DP$ property but without the $DP$ property and show that the completion of even normed spaces with the $DP$ property may not have the $DP$ property.
\end{abstract}


\end{frontmatter}

%%%%%%%%%%%%%%%%%%%%%%%%%%%%%%%%%%%%%%%%%%%
%%%%%%%%%%%%%%%%%%%%%%%%%%%%%%%%%%%%%%%%%%%
%%%%%%%%%%%%%%%%%%%%%%%%%%%%%%%%%%%%%%%%%%%
%%%%%%%%%%%%%%%%%%%%%%%%%%%%%%%%%%%%%%%%%%%
%%%%%%%%%%%%%%%%%%%%%%%%%%%%%%%%%%%%%%%%%%%

\section{Introduction}

%%%%%%%%%%%%%%%%%%%%%%%%%%%%%%%%%%%%%%%%%%%
%%%%%%%%%%%%%%%%%%%%%%%%%%%%%%%%%%%%%%%%%%%
%%%%%%%%%%%%%%%%%%%%%%%%%%%%%%%%%%%%%%%%%%%
%%%%%%%%%%%%%%%%%%%%%%%%%%%%%%%%%%%%%%%%%%%
%%%%%%%%%%%%%%%%%%%%%%%%%%%%%%%%%%%%%%%%%%%

One of the most important properties of Banach spaces is the property of being a Dunford--Pettis space introduced by Grothendieck in \cite{Grothen}.

\begin{definition}[\cite{Grothen}] \label{def:DP-Banach}{\em
A Banach space $X$ is said to have the {\em Dunford--Pettis property} ({\em $DP$ property}) if every weakly compact operator $T$ from $X$ to any Banach space $Y$ is completely continuous (i.e., $T$ sends weakly convergent sequences from X to norm-convergent sequences in Y). \qed}
\end{definition}
It is well known that all Banach $C(K)$-spaces and $L_1$-spaces have the $DP$ property.

In \cite{Grothen}, Grothendieck defined also the Dunford--Pettis property and the strict Dunford--Pettis property in the realm of all locally convex spaces. For more details and historical remarks we refer the reader to  Section 9.4 of \cite{Edwards}.
%Let us recall the classical  definitions of the Dunford--Pettis property and the strict Dunford--Pettis property for locally convex spaces introduced by Grothendieck \cite{Grothen}, see also Section 9.4 of \cite{Edwards}.

\begin{definition}[\cite{Grothen}] \label{def:DP}{\em
A locally convex space $E$ is said to have
\begin{enumerate}
\item[$\bullet$] the {\em Dunford--Pettis  property} ($DP$ {\em property}) if each operator from $E$ into a Banach space $L$, which transforms bounded sets into relatively weakly compact sets, transforms each absolutely convex weakly compact set into a relatively compact subset of $L$;
\item[$\bullet$] the {\em strict Dunford--Pettis  property} ($SDP$ {\em property}) if each operator from $E$ into a Banach space $L$, which transforms bounded sets into relatively weakly compact sets, transforms each weakly Cauchy sequence in $E$  into a convergent (or, equivalently, Cauchy) sequence in $L$. \qed
\end{enumerate}  }
\end{definition}

Let $E$ be a locally convex space  (lcs for short), and let $E'_\beta$ be the topological dual $E'$ of $E$ endowed with the strong topology $\beta(E',E)$. Denote by $\Sigma'(E')$ the family of all absolutely convex, equicontinuous, weakly compact subsets of $E'_\beta$, and let $\tau_{\Sigma'}$ be the topology on $E$ of uniform convergence on the elements of $\Sigma'(E')$. In what follows we shall call the topology $\tau_{\Sigma'}$ by the {\em Grothendieck topology}. The next characterizations of the $DP$ property and the strict $DP$ property are proved by Grothendieck in Propositions 1 and 1 bis of \cite{Grothen} (see also Theorem 9.3.4 and page 633 of \cite{Edwards}).
\begin{theorem}[\cite{Grothen}] \label{t:DP}
Let $E$ be a locally convex space. Then:
\begin{enumerate}
\item[{\rm(i)}]   $E$ has the $DP$ property if and only if every absolutely convex, weakly compact subset of $E$ is precompact for the Grothendieck topology $\tau_{\Sigma'}$;
\item[{\rm(ii)}]   $E$ has the strict $DP$ property if and only if each weakly Cauchy sequence in $E$ is Cauchy for the Grothendieck topology $\tau_{\Sigma'}$.
\end{enumerate}
\end{theorem}

In Proposition~2 of \cite{Grothen}, Grothendieck proved that a Banach space $E$ has the $DP$ property if and only if for every weakly null sequences $\{ x_n\}_{n\in\w}$ and $\{ \chi_n\}_{n\in\w}$ in $E$ and  the Banach dual $E'$ of $E$, respectively, it follows that $\lim_{n\to \infty} \langle\chi_n,x_n\rangle=0$. (He used this result to show that every Banach space $C(K)$ has the $DP$ property, see Th\'{e}or\`{e}me~1 of \cite{Grothen}.)
Albanese, Bonet and Ricker (see Corollary~3.4 of \cite{ABR}) generalized  Grothendieck's result by  proving  that if $E$ is a Fr\'{e}chet space (or even  a  strict $(LF)$-space), then $E$ has the $DP$ property if and only if $\lim_{n\to \infty} \langle\chi_n,x_n\rangle=0$ for any weakly null sequences $\{ x_n\}_{n\in\w}$ and $\{ \chi_n\}_{n\in\w}$ in $E$ and $E'_\beta$, respectively.
These results motivate us to introduce in \cite{Gabr-free-resp}  the following  ``sequential'' version of the $DP$ property in the class of all  locally convex spaces.

\begin{definition}[\cite{Gabr-free-resp}] \label{def:sDP}{\em
A locally convex space $E$ is said to have  the {\em sequential Dunford--Pettis property} ($sDP$ {\em property})  if for every weakly null sequences $\{ x_n\}_{n\in\w}$ and $\{ \chi_n\}_{n\in\w}$ in $E$ and  the strong dual $E'_\beta$ of $E$, respectively, it follows that  $\lim_{n\to \infty} \langle\chi_n,x_n\rangle=0$.\qed}
\end{definition}
Then Proposition~3.3 of \cite{ABR} can be formulated as follows: every barrelled quasi-complete space with the  $DP$ property has also the sequential $DP$ property. However, there are locally convex spaces with the $DP$ property but without the sequential $DP$ property, see Theorem 1.7 of \cite{Gabr-free-lcs}.

Let $X$ be a Tychonoff space. Denote by $C_p(X)$ and $\CC(X)$ the space $C(X)$ of continuous functions on $X$ endowed with the pointwise topology and the compact-open topology, respectively. Another important class of function spaces on $X$ widely studied in General Topology and Functional Analysis is the class $B_\alpha(X)$ of Baire-$\alpha$ functions. For $\alpha=0$, we put $B_0(X):=C_p(X)$. For every nonzero countable ordinal $\alpha$, let $B_\alpha(X)$ be the family of all functions $f:X\to \IF$ that are pointwise limits of sequences $\{f_n\}_{n\in\w}\subseteq \bigcup_{\beta<\alpha}B_\beta(X)$. All the spaces $B_\alpha(X)$ are endowed with the topology of pointwise convergence, inherited from the Tychonoff product $\IF^X$, where $\IF$ is the field of real or complex numbers. In the next theorem we summarize some known results (although the clause (iii) is proved in 9.4.6(a) of \cite{Edwards} only for normal spaces, the analysis of the proof for the spaces $C_p(X)$ and $\CC(X)$ shows that the assertion remains true for any Tychonoff space $X$).
\begin{theorem} \label{t:function-DP}
Let $X$ be a Tychonoff space, and let $\alpha$ be a  nonzero countable ordinal. Then:
\begin{enumerate}
\item[{\rm(i)}] {\rm(\cite[9.4.6(a)]{Edwards})} $C_p(X)$ has the $DP$ property and the strict $DP$ property;
\item[{\rm(ii)}] {\rm(\cite[Theorem~1.6]{GK-DP})} $C_p(X)$ has the sequential $DP$ property;
\item[{\rm(iii)}] {\rm(\cite[9.4.6(a)]{Edwards})} $\CC(X)$ has the $DP$ property and the strict $DP$ property;
\item[{\rm(iv)}] {\rm(\cite[Theorem~1.7]{GK-DP})} if $X$ is an ordinal space $[0,\kappa)$ or a locally compact paracompact space, then $\CC(X)$ has the sequential $DP$ property;
\item[{\rm(v)}] {\rm(\cite[Theorem~3.6]{BG-Baire-lcs})} $B_\alpha(X)$ has the $DP$ property and the sequential $DP$ property.
\end{enumerate}
\end{theorem}

Let $p\in[1,\infty]$, and let $E$ be a locally convex space. Recall (see Section 19.4 in \cite{Jar}) that a sequence $\{x_n\}_{n\in\w}$ in $E$ is called {\em weakly $p$-summable} if  for every $\chi\in E'$ it follows that $(\langle\chi,x_n\rangle)\in \ell_p$ if $p\in[1,\infty)$ and $(\langle\chi,x_n\rangle)\in c_0$ if $p=\infty$.

Let $p\in[1,\infty]$, and let $X$ and $Y$ be Banach spaces.  Generalizing the notion of completely continuous operators   Castillo and S\'{a}nchez  defined in \cite{CS} an operator $T:X\to Y$ to be {\em $p$-convergent} if $T$ sends weakly $p$-summable sequences of $X$ into null-sequences of $Y$. In their  fundamental article  \cite{CS},  Castillo and S\'{a}nchez defined the $p$-version of the Dunford--Pettis property as follows.

\begin{definition}[\cite{CS}] \label{def:DP-p-Banach}{\em
Let $p\in[1,\infty]$. A Banach space $X$ is said to have the {\em Dunford--Pettis property of order $p$} (the {\em $DP_p$ property}) if for each Banach space $Y$, every weakly compact operator $T:X\to Y$ is $p$-convergent.\qed}
\end{definition}
So, a Banach space has the $DP$ property if and only if it has the $DP_\infty$ property.

Generalizing the aforementioned Proposition 2 of Grothendieck \cite{Grothen}, the following characterization of Banach spaces with the  $DP_p$ property was obtained in Proposition 3.2 of  \cite{CS}.

\begin{proposition}[\cite{CS}] \label{p:Banach-DPp}
Let $p\in[1,\infty]$. For a Banach space $X$ the following assertions are equivalent:
\begin{enumerate}
\item[{\rm(i)}] $X$ has the  $DP_p$ property;
\item[{\rm(ii)}] if $\{x_n\}_{n\in\w}$ is a weakly $p$-summable sequence  in $X$ and  $\{\chi_n\}_{n\in\w}$ is a weakly null-sequence  in $X^\ast$, then $\langle\chi_n,x_n\rangle\to 0$;
\item[{\rm(iii)}] every weakly compact operator $T$ from $X$ to a Banach space $Y$ transforms weakly $p$-compact sets of $X$ into norm compact sets of $Y$.
\end{enumerate}
\end{proposition}
In \cite{CS} it is  also constructed numerous distinguished examples of Banach spaces with or without the  $DP_p$ property. Since the Dunford--Pettis property of order $p$ for Banach spaces and especially Banach lattices was intensively studied by many authors, see for example \cite{ABEE,EH,Ghenciu-pGP,GhLe,NLE}.
A more general notion than the $DP_p$ property was proposed by Karn and Sinha \cite{Karn-Sinha}.

Let $p\in[1,\infty]$, and let $E$ and $L$ be locally convex spaces. Extending the notion of $p$-convergent operators between Banach spaces to the general case, we defined in \cite{Gab-Pel} an operator $T:E\to L$ to be {\em $p$-convergent} if $T$ sends weakly $p$-summable sequences of $E$ into null-sequences of $L$.

All the above mentioned definitions and results motivate us to introduce the following Dunford--Pettis type properties in the realm of all locally convex spaces.
\begin{definition} \label{def:DP-general}{\em
Let $p,q\in[1,\infty]$. A locally convex space $E$ is said to have
\begin{enumerate}
\item[$\bullet$] the {\em strict Dunford--Pettis property of order $p$} (the {\em strict $DP_p$ property} or the {\em $SDP_p$ property}) if for each Banach space $L$, every operator $T:E\to L$, which transforms bounded sets into relatively weakly compact sets,  is $p$-convergent;
%\item[$\bullet$] the {\em weak Dunford--Pettis property of order $p$} ({\em weak $DP_p$ property}) if for each Banach space $L$, any weakly compact operator $T\in\LL(E,L)$ is $p$-convergent;
\item[$\bullet$] the {\em sequential Dunford--Pettis property of order $(p,q)$} (the {\em sequential $DP_{(p,q)}$ property} or the {\em $sDP_{(p,q)}$ property}) if
    \[
      \lim_{n\to \infty} \langle\chi_n,x_n\rangle=0
    \]
    for every weakly $p$-summable sequence $\{ x_n\}_{n\in\w}$ in $E$ and each weakly $q$-summable sequence $\{ \chi_n\}_{n\in\w}$ in  $E'_\beta$. If $p=q$ we shall say that $E$ has the {\em sequential $DP_{p}$ property}. \qed
%\item[$\bullet$] the [{\em strong}] {\em $DP^\ast$ property of order $(p,q)$} (or [{\em strong}] {\em $DP_{(p,q)}^\ast$ property}) if every weakly sequentially $p$-compact (resp., weakly sequentially $p$-precompact) set in $E$ is $q$-limited;
%\item[$\bullet$] the [{\em strong}] {\em $DPV^\ast$ property of order $(p,q)$} (or [{\em strong}] {\em $DPV_{(p,q)}^\ast$ property}) if every weakly sequentially $p$-compact (resp., weakly sequentially $p$-precompact) set in $E$ is a $q$-$(V^\ast)$ subset of $E$.
\end{enumerate}
} %If $p=q$ we shall say that $E$ has the {\em sequential $DP_{p}$ property}. \qed }%, the [{\em strong}] {\em $DP_p^\ast$ property} or the  [{\em strong}] {\em $DPV_p^\ast$ property}, respectively.  In the case when  $p=q=\infty$ we shall say that $E$ has the {\em sequential $DP$ property}, the [{\em strong}] {\em $DP^\ast$ property} or the  [{\em strong}] {\em $DPV^\ast$ property}, respectively. \qed }
\end{definition}
It is clear that  $E$ has the strict $DP_{\infty}$ property if and only if it has the strict $DP$ property, and $E$ has the  sequential $DP_{\infty}$ property  if and only if it has the sequential $DP$ property.

The purpose of the article is to study the  $DP$ property, the strict $DP_{p}$ property and the sequential $DP_{(p,q)}$ property in the class of all locally convex spaces.

Now we describe the content of the article in detail.
In Section \ref{sec:prel} we fix the main notions, recall some known results (some of them are taken from our recent article \cite{Gab-Pel}) and prove several necessary preliminary results.

One of the most important classes of topological spaces is  the class of angelic spaces. By the classical Eberlein--\v{S}mulian theorem  Banach spaces with the weak topology are angelic. Generalizing this theorem Grothendieck showed that the space $C_p(K)$ over any compact space $K$ is angelic. %The smaller class of strictly angelic spaces defined by Govaerts  \cite{Gov} is countable productive.
To get natural relationships between Dunford--Pettis type properties we define and study  in Section \ref{subsec:p-angelic}
weakly (sequentially) $p$-angelic spaces for every $p\in[1,\infty]$. In Proposition \ref{p:Lp-q-angelic} we show that if $1< p<\infty$ and $q\in[1,\infty]$, then $\ell_p$ is weakly sequentially $q$-angelic if and only if it is weakly $q$-angelic if and only if $q\geq p^\ast$, where $\tfrac{1}{p}+\tfrac{1}{p^\ast}=1$.

%In Section \ref{sec:general-DP} we give exact formulations and detailed proofs of three general theorems extending numerous results proved earlier by many specialists, see   Grothendieck \cite{Grothen} and especially Theorems 9.2.1 and 9.3.1 in the book of Edwards \cite{Edwards}.

In Section \ref{sec:char-DP} we characterize locally convex spaces with the Dunford--Pettis type property, the strict $DP_p$ property and the sequential $DP_{(p,q)}$ property. It turns out that these properties can be characterized using the following Schur type properties, were $E_w$ denotes an lcs $E$ endowed with the weak topology.
\begin{definition} \label{def:weak-Glick} {\em
Let $p\in[1,\infty]$. A locally convex space $E$ is said to have a
\begin{enumerate}
\item[$\bullet$] the {\em Glicksberg property} if $E$ and $E_w$ have the same  compact sets;
\item[$\bullet$] a {\em weak Glicksberg property} if $E$ and $E_w$ have the same absolutely convex compact sets;
\item[$\bullet$] the {\em Schur property} if $E$ and $E_w$ have the same convergent sequences;
\item[$\bullet$] the {\em $p$-Schur property} if every weakly $p$-summable sequence in $E$ is a null sequence.\qed
\end{enumerate}}
\end{definition}
The Glicksberg property and the Schur property are widely studied in topology and analysis, see for example \cite{Gabr-free-resp} and \cite{Gab-Respected} and references therein. For $p\in[1,\infty]$, Banach spaces with the $p$-Schur property was defined independently by Deghani and Moshtaghioun \cite{DM} and Fourie and Zeekoei \cite{FZ-pL}, and the general case was considered in \cite{Gab-Pel}. In Theorem \ref{t:def-DP} we show that an lcs $E$ has the $DP$ property if and only if the space $E$ endowed with the Grothendieck topology $\tau_{\Sigma'}$ defined above has the weak Glicksberg property, and Theorem \ref{t:strict-DPp} states that $E$ has the strict $DP_p$ property if and only if the space $(E,\tau_{\Sigma'}) $ has the $p$-Schur property. In this section we also construct numerous (counter)examples which show that all Dunford--Pettis type properties defined in Definitions  \ref{def:DP} and \ref{def:DP-general}   are different. In particular, we give the first example of a locally convex space with the strict $DP$ property but without the $DP$ property, see
Example \ref{exa:SDP-non-DP}. On  the other hand, Example \ref{exa:weak-Glick-non-Glick} shows that there are locally convex spaces with  the weak Glicksberg property but without even the Schur property.

In Section \ref{sec:perm-DP} we study permanent properties of Dunford--Pettis type properties. In Proposition \ref{p:weak-top-DP} we show that if $E=E_w$, then $E$ has  the $DP$ property, the strict $DP_p$ property, and if in addition $E$ is quasibarrelled, then $E$ has also the sequential $DP_{(p,q)}$ property. In this proposition we show also that if $E$ or its strong dual $E'_\beta$ is feral (an lcs is called {\em feral} if  all its bounded subsets are finite-dimensional),  then $E$ has the sequential $DP_{(p,q)}$ property. As a consequence in Corollary \ref{c:DP-weaker} we generalize (i), (ii) and (v) of Theorem \ref{t:function-DP} by showing that any subspace of $\IF^X$ containing $C_p(X)$ has the $DP$ property, the strict $DP_p$ property and the sequential $DP_{(p,q)}$ property. In Proposition \ref{p:DP-product} we show that arbitrary (Tychonoff) products and locally convex direct sums of spaces with the $DP$ property, the strict $DP_p$ property or the sequential $DP_{(p,q)}$ property have the same property. On the other hand, quotients  and completions of locally convex spaces with the Dunford--Pettis property may not have the $DP$ property, see Proposition \ref{p:L(X)-DP} and Proposition \ref{p:Lp-DP}.
In Proposition \ref{p:DP=>seqDP} we consider relationships between the Dunford--Pettis type properties for the important case of locally complete spaces. In particular, we show that (1) a locally complete space $E$ with the $DP$ property has also the strict $DP_p$ property, and the converse is true if additionally $E$ is a weakly sequentially $p$-angelic space, and (2) if a locally complete space $E$ is such that it is weakly sequentially $p$-angelic and $E'_\beta$ is a weakly sequentially $q$-angelic space, then the sequential $DP_{(p,q)}$ property implies the $DP$ property and  the strict $DP_p$ property.
Consequently we essentially extend the aforementioned  Proposition 3.3 of \cite{ABR} %(which states that if $E$ is a barrelled and quasi-complete lcs with the $DP$ property, then $E$ has the sequential $DP$ property)
by showing that if a quasibarrelled locally complete space $E$ has the $DP$ property, then $E$ has  the strict $DP_p$ property and the sequential $DP_{(p,q)}$ property for all  $p,q\in[1,\infty]$, see Corollary \ref{c:DP=>sDPp}.

%%%%%%%%%%%%%%%%%%%%%%%%%%%%%%%%%%%%%
%%%%%%%%%%%%%%%%%%%%%%%%%%%%%%%%%%%%%
%%%%%%%%%%%%%%%%%%%%%%%%%%%%%%%%%%%%%
%%%%%%%%%%%%%%%%%%%%%%%%%%%%%%%%%%%%%
%%%%%%%%%%%%%%%%%%%%%%%%%%%%%%%%%%%%%

\section{Preliminaries results and notations} \label{sec:prel}

%%%%%%%%%%%%%%%%%%%%%%%%%%%%%%%%%%%%%
%%%%%%%%%%%%%%%%%%%%%%%%%%%%%%%%%%%%%
%%%%%%%%%%%%%%%%%%%%%%%%%%%%%%%%%%%%%
%%%%%%%%%%%%%%%%%%%%%%%%%%%%%%%%%%%%%
%%%%%%%%%%%%%%%%%%%%%%%%%%%%%%%%%%%%%

%%%%%%%%%%%%%%%%%%%%%%%%%%%%%%%%%%%%%
%%%%%%%%%%%%%%%%%%%%%%%%%%%%%%%%%%%%%
%%%%%%%%%%%%%%%%%%%%%%%%%%%%%%%%%%%%%
%%%%%%%%%%%%%%%%%%%%%%%%%%%%%%%%%%%%%
%%%%%%%%%%%%%%%%%%%%%%%%%%%%%%%%%%%%%

%\subsection{Some general results and notations} \label{subsec:prel}

%%%%%%%%%%%%%%%%%%%%%%%%%%%%%%%%%%%%%
%%%%%%%%%%%%%%%%%%%%%%%%%%%%%%%%%%%%%
%%%%%%%%%%%%%%%%%%%%%%%%%%%%%%%%%%%%%
%%%%%%%%%%%%%%%%%%%%%%%%%%%%%%%%%%%%%
%%%%%%%%%%%%%%%%%%%%%%%%%%%%%%%%%%%%%

We start with some necessary definitions and notations used in the article. Set  $\w:=\{ 0,1,2,\dots\}$.
All topological spaces are assumed to be Tychonoff (= completely regular and $T_1$). The closure of a subset $A$ of a topological space $X$ is denoted by $\overline{A}$, $\overline{A}^X$ or $\cl_X(A)$.

Let $(E,\tau)$ be a locally convex space. The span of a subset $A$ of $E$ and its closure are denoted by $\spn(A)$ and $\cspn(A)$, respectively. We denote by $\Nn_0(E)$ (resp., $\Nn_{0}^c(E)$) the family of all (resp., closed absolutely convex) neighborhoods of zero of $E$. The family of all bounded subsets of $E$ is denoted by $\Bo(E)$. The topological dual space of $E$  is denoted by $E'$. The value of $\chi\in E'$ on $x\in E$ is denoted by $\langle\chi,x\rangle$ or $\chi(x)$.  We denote by $E_w$ and $E_\beta$ the space $E$ endowed with the weak topology $\sigma(E,E')$ or with the strong topology $\beta(E,E')$, respectively. The topological dual space $E'$ of $E$ endowed with weak$^\ast$ topology $\sigma(E',E)$ or with the strong topology $\beta(E',E)$ is denoted by $E'_{w^\ast}$ or $E'_\beta$, respectively. The closure of a subset $A$ in the weak topology is denoted by $\overline{A}^{\,w}$ or $\overline{A}^{\,\sigma(E,E')}$, and $\overline{B}^{\,w^\ast}$ (or $\overline{B}^{\,\sigma(E',E)}$) denotes the closure of $B\subseteq E'$ in the weak$^\ast$ topology. The {\em polar} of a subset $A$ of $E$ is denoted by
\[
A^\circ :=\{ \chi\in E': \|\chi\|_A \leq 1\}, \quad\mbox{ where }\quad\|\chi\|_A=\sup\big\{|\chi(x)|: x\in A\cup\{0\}\big\}.
\]
A subset $B$ of $E'$ is {\em equicontinuous} if $B\subseteq U^\circ$ for some $U\in \Nn_0(E)$.
A locally convex vector topology $\TTT$ on $E$ is called {\em compatible with $\tau$} if $(E,\tau)'=(E,\TTT)'$ algebraically. It is well known that there is a finest locally convex vector topology $\mu(E,E')$  compatible with $\tau$. The topology $\mu(E,E')$ is called the {\em Mackey topology}, and if $\tau=\mu(E,E')$, the space $E$ is called a {\em Mackey space}. Set $E_\mu:=\big(E,\mu(E,E')\big)$. Any bornology $\AAA(E)$ on $E$  defines the {\em $\AAA$-topology $\TTT_\AAA$} on the dual space $E'$, i.e. $\TTT_\AAA$ is the topology  of uniform convergence on the members of $\AAA$. We write $E'_\AAA :=(E', \TTT_\AAA)$.

We need the next easy lemma.
\begin{lemma} \label{l:polar-open}
If $U$ is a neighborhood of zero in a locally convex space $E$, then $U^\circ$ is a complete, absolutely convex and bounded subset of $E'_\beta$.
\end{lemma}

\begin{proof}
Clearly, $U^\circ$ is absolutely convex, and it is well known that $U^\circ$ is a  bounded subset of $E'_\beta$. We show that $U^\circ$ is complete under the strong topology. By the Alaoglu theorem, $U^\circ$ is weak$^\ast$-compact. Since the strong topology is finer than the weak$^\ast$ topology it follows that $U^\circ$ is strongly closed. %To show that $U^\circ$ is strongly bounded, let $A$ be a bounded subset of $E$ and $a>0$ be such that $A\subseteq aU$. Then $U^\circ \subseteq aA^\circ$ and hence $U^\circ$ is strongly bounded. % and hence $U^\circ$ is a precompact subset of  $(E'_\beta)_w$. %Indeed, if $A$ is a bounded subset of $E$ and $a>0$ is such that $A\subseteq aU$, it follows that $U^\circ \subseteq aA^\circ$ and hence $U^\circ$ is strongly bounded.
Taking into account that the strong topology is a polar topology, Theorem 3.2.4 of \cite{Jar} implies that $U^\circ$ is a complete subset of $E'_\beta$. \qed
\end{proof}

Let $E$ and $L$ be locally convex spaces. If $T:E\to L$ is a weakly (= weak-weak) continuous linear map, we denote by $T^\ast:L'\to E'$ the adjoint linear map defined by $\langle T^\ast(\chi),x\rangle=\langle \chi, T(x)\rangle$ for every $\chi\in L'$ and each $x\in E$. The family of all  operators (= continuous linear maps) from $E$ to $L$ is denoted by $\LL(E,L)$.

Recall that an lcs $E$ is called {\em semi-reflexive} if the canonical map $J_E:E\to E''=(E'_\beta)'_\beta$ defined by  $\langle J_E(x),\chi\rangle:=\langle\chi,x\rangle$ ($\chi\in E'$) is an algebraic isomorphism; if in addition $J_E$ is a topological isomorphism the space $E$ is called {\em reflexive}. Recall also that an lcs $E$ is called {\em semi-Montel} if every bounded subset of $E$ is relatively compact.

Recall that a locally convex space  $E$
\begin{enumerate}
\item[$\bullet$] is {\em sequentially complete} if each Cauchy sequence in $E$ converges;
\item[$\bullet$]  is {\em locally complete} if the closed absolutely convex hull of a null sequence in $E$ is compact;
\item[$\bullet$] has the {\em Krein property} if the closed absolutely convex hull of each weakly compact subset $K$ of $E$ is weakly compact.
\end{enumerate}
Every sequentially complete space is locally complete, but the converse is not true in general.
The Krein property was defined and characterized in \cite{Gabr-free-resp} being motivated by the Krein theorem \cite[\S~24.5(4)]{Kothe}. We shall use repeatedly the next result proved in Theorem 10.2.4 of \cite{Jar}.

%Although the next theorem is proved in Theorem 10.2.4 of \cite{Jar}  we give its direct and shorter proof to make the paper self-contained.
\begin{theorem} \label{t:weakly-p-lc}
For an lcs $(E,\tau)$ the following assertions are equivalent:
\begin{enumerate}
\item[{\rm(i)}] $E$ is locally complete;
\item[{\rm(ii)}] for some $($every$)$ $p\in[1,\infty]$, the closed absolutely convex hull of any weakly $p$-summable sequence in $E$ is weakly compact.
\end{enumerate}
\end{theorem}

\begin{lemma} \label{l:weak-Glick}
If a locally convex space $(E,\tau)$ has the Glicksberg property, then it has the weak Glicksberg property; the converse is true if $E$ has the Krein property.
\end{lemma}

\begin{proof}
The first assertion is clear. Assume that $E$ has the Krein property. Then for every weakly compact subset $K$ of $E$ its closed absolutely convex hull $\cacx(K)$ is also weakly compact. By the weak Glicksberg property, $\cacx(K)$ and also its closed subspace $K$ are $\tau$-compact. Thus $E$ has the Glicksberg property.\qed
\end{proof}

Recall that a subset $A$ of a topological space $X$ is called
\begin{enumerate}
\item[$\bullet$] {\em sequentially precompact} if every sequence in $A$ has a Cauchy subsequence;
\item[$\bullet$] ({\em relatively}) {\em sequentially compact}  if each sequence in $A$ has a subsequence converging to a point of $A$ (resp., of $X$).
\end{enumerate}

We shall use the next simple lemmas, see Lemma 3.5 and Lemma 17.18 of \cite{Gab-Pel}, respectively.

\begin{lemma} \label{l:seq-p-comp}
A subset $A$ of a sequentially complete locally convex space $E$ is sequentially precompact if and only if it is relatively sequentially compact in $E$.
\end{lemma}

\begin{lemma} \label{l:weak-to-norm}
Let $E$ be a locally convex space such that $E=E_w$, and let $L$ be a normed space. Then every $T\in\LL(E,L)$ is finite-dimensional.
\end{lemma}

%\begin{proof}
%Observe that $T$ can be extended to an operator ${\bar T}$ from a completion ${\bar E}$ of $E$ to a completion ${\bar L}$ of $L$. As $E$ carries its weak topology, we obtain ${\bar E}=\IF^\kappa$ for some cardinal $\kappa$. Since ${\bar T}$ is continuous, there is a finite subset $\lambda$ of $\kappa$ such that ${\bar T}\big(\{0\}^\lambda \times \IF^{\kappa\SM \lambda} \big)$ is contained in the unit ball $B_{{\bar L}}$ of ${\bar L}$. Taking into account that $B_{{\bar L}}$ contains no non-trivial linear subspaces we obtain that $\{0\}^\lambda \times \IF^{\kappa\SM \lambda}$ is contained in the kernel $\ker({\bar T})$ of ${\bar T}$. Therefore ${\bar T}[\IF^\kappa]={\bar T}[\IF^\lambda]$ is finite-dimensional. Thus also $T$ is finite-dimensional.\qed
%\end{proof}

%We shall use  the next simple lemma.

The following statement is Proposition 7.9 of \cite{Gab-Pel}.
\begin{proposition} \label{p:strong-dual-feral}
A locally convex space $E$  carries its weak topology and is quasibarrelled if and only if $E'_\beta$ is feral.
\end{proposition}

%\begin{proof}
%Assume that $E$  carries its weak topology and is quasibarrelled. If $B$ is a bounded subset of $E'_\beta$, then $B$ is equicontinuous. Since $E$ carries its weak topology, there are a finite subset $F$ of $E'$ and $\e>0$ such that $B\subseteq [F;\e]^\circ$ and hence $B$ is contained in the finite-dimensional subspace $\spn(F)$.

%Conversely, assume that $E'_\beta$ is feral. Then every bounded subset of $E'_\beta$, being finite-dimensional, is trivially equicontinuous. Therefore $E$ is quasibarrelled. In particular, $E$ is a subspace of $E''=(E'_\beta)'_\beta$. But since all bounded subsets of $E'_\beta$ are finite-dimansional it follows that $E''$ carries its weak$^\ast$ topology. Thus $E$ carries its weak topology.\qed
%\end{proof}

Let $p\in[1,\infty]$. We recall that a sequence  $\{x_n\}_{n\in\w}$ in an lcs $E$ is called
\begin{enumerate}
%\item[$\bullet$]  {\em weakly $p$-summable} if  for every $\chi\in E'$, it follows
%\[
%\mbox{$(\langle\chi, x_n\rangle)_{n\in\w} \in\ell_p$ if $p<\infty$, and  $(\langle\chi, x_n\rangle)_{n\in\w} \in c_0$ if $p=\infty$;}
%\]
\item[$\bullet$] {\em weakly $p$-convergent to $x\in E$} if  $\{x_n-x\}_{n\in\w}$ is weakly $p$-summable;
\item[$\bullet$] {\em weakly $p$-Cauchy} if for each pair of strictly increasing sequences $(k_n),(j_n)\subseteq \w$, the sequence  $(x_{k_n}-x_{j_n})_{n\in\w}$ is weakly $p$-summable.\qed
\end{enumerate}
If $1\leq p<\infty$, we denote by $\ell^w_p(E)$ the family of all weakly $p$-summable sequences in $E$. If $p=\infty$, for the sake of simplicity we denote by $\ell^w_\infty(E):=c_0^w(E)$  the space of all weakly null-sequences in $E$. For numerous results concerning the spaces $\ell^w_p(E)$ see \cite{Gab-Pel}.

Analogously to the corresponding notions in the case of Banach spaces the following  weak $p$-versions of compact-type properties in topological vector spaces (tvs for short) were defined in \cite{Gab-Pel}.
Let $p\in[1,\infty]$. A subset   $A$ of a separated tvs $E$ is called
\begin{enumerate}
%\item[$\bullet$] {\em sequentially precompact} if every sequence in $A$ has a Cauchy subsequence;
\item[$\bullet$] {\em weakly  sequentially $p$-precompact} if every sequence from $A$ has a  weakly $p$-Cauchy subsequence;
\item[$\bullet$] ({\em relatively}) {\em weakly sequentially $p$-compact} if every sequence in $A$ has a weakly $p$-convergent  subsequence with limit in $A$ (resp., in $E$).\qed
\end{enumerate}

%\begin{definition}\label{def:lcs-p-oper}{\em
Let $p\in[1,\infty]$.  Recall (see \cite{Gab-Pel}) that  a linear map $T:E\to L$ between separated topological vector spaces $E$ and $L$ is said to be
\begin{enumerate}
%\item[$\bullet$]  {\em weakly $(p,q)$-convergent} if it sends weakly $p$-summable sequences into  weakly $q$-summable sequences;
\item[$\bullet$]  {\em $p$-convergent} if it sends weakly $p$-summable sequences into null-sequences;
\item[$\bullet$]  {\em weakly sequentially $p$-compact} if  for some $U\in \Nn_0(E)$, the set $T(U)$ is a relatively weakly  sequentially $p$-compact subset of $L$;
\item[$\bullet$]   {\em weakly  sequentially $p$-precompact} if  for some $U\in \Nn_0(E)$, the set $T(U)$ is   weakly  sequentially $p$-precompact in $L$.\qed
%\item[$\bullet$] {\em weakly  completely $(p,q)$-continuous} ({\em $(p,q)$-wcc} or a {\em $(p,q)$-Dieudonn\'{e} operator}) if it sends weakly $p$-Cauchy sequences into weakly $q$-convergent sequences;
%\item[$\bullet$] ({\em weakly}) {\em sequentially compact} if  for some $U\in \Nn_0(E)$, the set $T(U)$ is a relatively (weakly) sequentially compact set in $L$;
%\item[$\bullet$] ({\em weakly}) {\em  compact} if  for some $U\in \Nn_0(E)$, the set $T(U)$ is a relatively (weakly) compact set in $L$;
%\item[$\bullet$] ({\em weakly}) {\em precompact} if  for some $U\in \Nn_0(E)$, the set $T(U)$ is a (weakly) precompact set in $L$.
\end{enumerate}
%\end{definition}

We shall use the next two lemmas, see Lemma 13.2 and 13.4 in \cite{Gab-Pel}.

\begin{lemma} \label{l:image-p-seq-com}
Let $p\in[1,\infty]$, $E$ and $L$ be locally convex spaces, and let $T:E\to L$ be a weakly continuous linear map. If $A$ is a $($relatively$)$ weakly sequentially $p$-compact set {\rm(}resp., a weakly sequentially $p$-precompact set{\rm)} in $E$, then the image  $T(A)$ is a $($relatively$)$ weakly sequentially $p$-compact {\rm(}resp.,  weakly sequentially $p$-precompact{\rm)}  set in $L$.
\end{lemma}

\begin{lemma} \label{l:p-conver-p-Cauchy}
Let $p\in[1,\infty]$, $E$ and $L$ be locally convex spaces, and let $T:E\to L$  be a $p$-convergent linear map. If $\{x_n\}_{n\in\w}\subseteq E$ is weakly $p$-Cauchy, then the sequence $\big\{T(x_n)\big\}_{n\in\w}$ is Cauchy in $L$. Consequently, if $L$ is sequentially complete, then $\big\{T(x_n)\big\}_{n\in\w}$ converges in $L$.
\end{lemma}

The next characterization of $p$-convergent operators is given in Proposition 13.5 in \cite{Gab-Pel}.
\begin{proposition} \label{p:p-convergent-s}
Let $p\in[1,\infty]$, and let $T$ be a weak-weak sequentially continuous linear map from a locally convex space $E$ to a sequentially complete locally convex space $L$. Then the following assertions are equivalent:
\begin{enumerate}
\item[{\rm(i)}] $T$ is $p$-convergent;
\item[{\rm(ii)}] $T(A)$ is relatively sequentially compact in $L$ for each weakly sequentially $p$-precompact subset $A$ of $E$;
\item[{\rm(iii)}] $T(A)$ is sequentially precompact in $L$ for any weakly sequentially $p$-precompact subset $A$ of $E$.
\end{enumerate}
\end{proposition}

We denote by $\bigoplus_{i\in I} E_i$ and $\prod_{i\in I} E_i$  the locally convex direct sum and the topological product of a non-empty family $\{E_i\}_{i\in I}$ of locally convex spaces, respectively. If $\xxx=(x_i)\in \bigoplus_{i\in I} E_i$, then the set $\supp(\xxx):=\{i\in I: x_i\not= 0\}$ is called the {\em support} of $\xxx$. The {\em support}  of a subset $A$ of $\bigoplus_{i\in I} E_i$ is the set $\supp(A):=\bigcup_{a\in A} \supp(a)$. Also we shall consider elements $\xxx=(x_i) \in \prod_{i\in I} E_i$ as functions on $I$ and write $\xxx(i):=x_i$.
The next lemma is Lemma 4.25 of \cite{Gab-Pel}.

\begin{lemma} \label{l:support-p-sum}
Let  $p\in[1,\infty]$, and let $\{E_i\}_{i\in I}$  be a non-empty family of locally convex spaces. Then:
\begin{enumerate}
\item[{\rm(i)}] a sequence $\{\chi_n\}_{n\in\w} \subseteq \big(\prod_{i\in I} E_i\big)'_\beta$ is weakly $p$-summable if and only if its support $F:=\supp\{\chi_n\}$ is finite and the sequence $\{\chi_{n}(i)\}_{n\in\w}$ is weakly $p$-summable in $(E_i)'_\beta$ for every $i\in F$;
\item[{\rm(ii)}] a sequence $\{x_n\}_{n\in\w} \subseteq \prod_{i\in I} E_i$  is weakly $p$-summable  if and only if the sequence $\{x_{n}(i)\}_{n\in\w}$ is weakly  $p$-summable in $E_i$ for every $i\in I$.
\item[{\rm(iii)}] a sequence $\{\chi_n\}_{n\in\w} \subseteq \big(\bigoplus_{i\in I} E_i\big)'_\beta$ is weakly $p$-summable if and only if the sequence $\{\chi_{n}(i)\}_{n\in\w}$ is weakly $p$-summable in $(E_i)'_\beta$ for every $i\in F$;
\item[{\rm(iv)}] a sequence $\{x_n\}_{n\in\w} \subseteq \bigoplus_{i\in I} E_i$  is weakly $p$-summable  if and only if its support $F:=\supp\{x_n\}$ is finite and the sequence $\{x_{n}(i)\}_{n\in\w}$ is weakly  $p$-summable in $E_i$ for every $i\in F$.
\end{enumerate}
\end{lemma}

\begin{lemma} \label{l:seq-p-comp-prod}
Let $p\in[1,\infty]$, and let $\{E_i\}_{i\in I}$ be a non-empty family of locally convex spaces.
\begin{enumerate}
\item[{\rm(i)}]  A subset $A$ of $E=\bigoplus_{i\in I} E_i$ is $($relatively$)$ weakly  sequentially $p$-compact {\rm(}resp.,  weakly sequentially $p$-precompact{\rm)} if and only if the support of $A$ is finite and for every $i\in I$, its projection $\pi_i(A)$ into $E_i$ is  $($relatively$)$ weakly  sequentially $p$-compact {\rm(}resp.,  weakly sequentially $p$-precompact{\rm)}.
\item[{\rm(ii)}] If $I=\w$, then a subset $A$ of $E=\prod_{i\in \w} E_i$ is $($relatively$)$ weakly  sequentially $p$-compact {\rm(}resp.,  weakly sequentially $p$-precompact{\rm)} if and only if for every $i\in \w$, its projection $\pi_i(A)$ into $E_i$ is  $($relatively$)$ weakly  sequentially $p$-compact {\rm(}resp.,  weakly sequentially $p$-precompact{\rm)}.
\end{enumerate}
\end{lemma}

\begin{proof}
The necessity in both cases (i) and (ii) follows from Lemma \ref{l:image-p-seq-com} applied to the projection $T:=\pi_i$ onto the $i$th coordinate and, for the case (i), the fact that $A$ is bounded (so the support of $A$ is finite). To prove the sufficiency, we distinguish between the cases (i) and (ii).
\smallskip

(i) Since the support of $A$ is finite, the assertion follows from (ii).
\smallskip

(ii) First we consider the (relatively) weakly  sequentially $p$-compact case. Assume that $\pi_i(A)$ is (relatively) weakly  sequentially $p$-compact for every $i\in\w$. Let $\{\xxx_n\}_{n\in\w}$ be a sequence in $A$. We proceed by induction. For $i=0$, since $\pi_0(A)$ is (relatively) weakly  sequentially $p$-compact one can choose an infinite subset $J_0$ of $\w$ such that the subsequence $\{\xxx_n(0)\}_{n\in J_0}$ weakly $p$-converges to some $x(0)\in \pi_0(A)$ (resp., $x(0)\in E_0$). Assume that we have chosen an infinite subset $J_{k-1}$ of $J_{k-2}$   such that the subsequence $\{\xxx_n(k-1)\}_{n\in J_{n-1}}$ weakly $p$-converges to some $x(k-1)\in \pi_{k-1}(A)$ (resp., $x(k-1)\in E_{k-1}$).  Since $\pi_k(A)$ is (relatively) weakly  sequentially $p$-compact, we can  choose  an infinite subset $J_k$ of  $J_{k-1}$  such that the subsequence $\{\xxx_n(k)\}_{n\in J_k}$ weakly $p$-converges to some $x(k)\in \pi_k(A)$ (resp., $x(k)\in E_k$). For every $k\in\w$, choose $n_k\in J_k$ such that $n_k<n_{k+1}$. Set $\xxx:=\big(x(k)\big)_{k\in\w}\in E$. It remains to show that $\xxx_{n_k}$ weakly $p$-converges to $\xxx$. To this end, let $\chi=(\chi_i)\in E'=\bigoplus_{i\in\w} E'_i$ and denote by $F$ the finite support of $\chi$. Then, by construction, we have
\[
\big(\langle\chi,\xxx_{n_k}-\xxx\rangle\big)_{k\in\w}=\Big(\sum_{i\in F} \langle\chi_i,\xxx_{n_k}(i)-\xxx(i)\rangle\Big)_{k\in\w}\in \ell_p \;\; (\mbox{or } \in c_0 \; \mbox{ if $p=\infty$}),
\]
as desired. Thus $A$ is a (relatively) weakly  sequentially $p$-compact set.

Now we consider the weakly  sequentially $p$-precompact case. Assume that $\pi_i(A)$ are weakly  sequentially $p$-precompact for every $i\in\w$, and let $\{\xxx_n\}_{n\in\w}$ be a sequence in $A$. As in the previous case we can extract a subsequence $\{\xxx_{n_k}\}_{k\in\w}$ of $\{\xxx_n\}_{n\in\w}$ such that for every $j\in\w$, the sequence $\{\xxx_{n_k}(j)\}_{k\in\w}$ is weakly $p$-Cauchy in $E_j$. We show that $\{\xxx_{n_k}\}_{k\in\w}$  is weakly $p$-Cauchy in $E$. To this end, we show that for every strictly increasing sequence $\{j_s\}_{s\in\w}$ in $\w$, the sequence $\{\xxx_{n_{j_s}}-\xxx_{n_{j_{s+1}}}\}_{s\in\w}$ is weakly $p$-summable. Fix an arbitrary $\chi=(\chi_t)\in E'$, and let $m\in\w$ be such that $\chi_t=0$ for every $t>m$. Then
\[
\big(\langle\chi,\xxx_{n_{j_s}}-\xxx_{n_{j_{s+1}}}\rangle\big) =\sum_{t\leq m} \Big( \langle\chi_t,\xxx_{n_{j_s}}(t)-\xxx_{n_{j_{s+1}}}(t)\rangle\Big)\in \ell_p \;\; (\mbox{or } \in c_0 \; \mbox{if } p=\infty)
\]
because $\{\xxx_{n_k}(t)\}_{k\in\w}$ is weakly $p$-Cauchy in $E_t$ for every $t\leq m$.
\qed
\end{proof}

Let $p\in[1,\infty]$.  A locally convex space $E$ is called
\begin{enumerate}
\item[$\bullet$] ({\em quasi}){\em barrelled} if every $\sigma(E',E)$-bounded (resp., $\beta(E',E)$-bounded) subset of $E'$ is equicontinuous;
\item[$\bullet$] {\em $p$-{\rm(}quasi{\rm)}barrelled } if every weakly $p$-summable sequence in $E'_{w^\ast}$ (resp., in $E'_\beta$) is equicontinuous.
\end{enumerate}
$p$-(quasi)barrelled  spaces are defined and studied in  \cite{Gab-Pel}.

We denote by $\ind_{n\in \w} E_n$  the inductive limit of a (reduced) inductive sequence $\big\{(E_n,\tau_n)\big\}_{n\in \w}$  of locally convex spaces. If in addition $\tau_m{\restriction}_{E_n} =\tau_n$ for all $n,m\in\w$ with $n\leq m$, the inductive limit $\ind_{n\in \w} E_n$ is called {\em strict} and is denoted by $\SI E_n$. It is well known that $E:=\SI E_n$ is regular, i.e., every bounded subset of $E$ is contained in some $E_n$. For more details we refer the reader to Section 4.5 of \cite{Jar}. In the partial case when all spaces $E_n$ are Fr\'{e}chet, the strict inductive limit is called a {\em strict $(LF)$-space}. One of the most important examples of strict $(LF)$-spaces is  the space $\mathcal{D}(\Omega)$ of test functions over an open subset $\Omega$ of $\IR^n$. The strong dual $\mathcal{D}'(\Omega)$ of $\mathcal{D}(\Omega)$ is the space of distributions.

Let $p\in[1,\infty]$. Then $p^\ast$ is defined to be the unique element of $ [1,\infty]$ which satisfies $\tfrac{1}{p}+\tfrac{1}{p^\ast}=1$. For $p\in[1,\infty)$, the space $\ell_{p^\ast}$  is the dual space of $\ell_p$. We denote by $\{e_n\}_{n\in\w}$ the canonical unit basis of $\ell_p$, if $1\leq p<\infty$, or the canonical basis of $c_0$, if $p=\infty$. The canonical basis of $\ell_{p^\ast}$ is denoted by $\{e_n^\ast\}_{n\in\w}$. Denote by  $\ell_p^0$ and $c_0^0$ the linear span of $\{e_n\}_{n\in\w}$  in  $\ell_p$ or $c_0$ endowed with the induced norm topology, respectively.

One of the most important classes of locally convex spaces is the class of free locally convex spaces introduced by Markov in \cite{Mar}. The {\em  free locally convex space}  $L(X)$ over a Tychonoff space $X$ is a pair consisting of a locally convex space $L(X)$ and  a continuous map $i: X\to L(X)$  such that every  continuous map $f$ from $X$ to a locally convex space  $E$ gives rise to a unique continuous linear operator $\Psi_E(f): L(X) \to E$  with $f=\Psi_E(f) \circ i$. The free locally convex space $L(X)$ always exists and is essentially unique, and $X$ is the Hamel basis of $L(X)$. So, each nonzero $\chi\in L(X)$ has a unique decomposition $\chi =a_1 i(x_1) +\cdots +a_n i(x_n)$, where all $a_k$ are nonzero and $x_k$ are distinct. The set $\supp(\chi):=\{x_1,\dots,x_n\}$ is called the {\em support} of $\chi$. In what follows we shall identify $i(x)$ with $x$ and consider $i(x)$ as the Dirac measure $\delta_x$ at the point $x\in X$. We also recall that $C_p(X)'=L(X)$ and $L(X)'=C(X)$. The space $L(X)$ has the Glicksberg property for every Tychonoff space $X$, and if $X$ is non-discrete, then $L(X)$ is not a Mackey space, see \cite{Gab-Respected} and \cite{Gabr-L(X)-Mackey}, respectively.

%We shall use also the following well known description of relatively compact subsets of $\ell_p$ and $c_0$,  see \cite[p.~6]{Diestel}.
%\begin{proposition} \label{p:compact-ell-p}
%{\rm(i)} A bounded subset $A$ of $\ell_p$, $p\in[1,\infty)$, is relatively compact if and only if
%\[
%\lim_{m\to\infty} \sup\Big\{ \sum_{m\leq n} |x_n|^p : x=(x_n)\in A\Big\} =0.
%\]
%{\rm(ii)} A bounded subset $A$ of $c_0$ is relatively compact if and only if
%\[
%\lim_{n\to\infty} \sup\{ |x_n|: x=(x_n)\in A\} =0.
%\]
%\end{proposition}

%The next notion will play a crucial role to characterize the Dunford--Pettis property in locally convex spaces.
%\begin{definition} \label{def:weak-Glick} {\em
%A locally convex space $E$ is said to have a {\em weak Glicksberg property} if $E$ and $E_w$ have the same absolutely convex compact sets.\qed}
%\end{definition}
% Our motivation to introduce this notion are Proposition \ref{p:weak-Glick-DP} and Example \ref{exa:weak-Glick-non-Glick}.
%which show that the next simple assertion It is clear that any lcs with the Glicksberg property has also the weak Glicksberg property, but the converse is not true in general.

We shall use the next assertion, see Proposition 6.5 of \cite{Gab-Pel}.
\begin{proposition} \label{p:Lp-Schur}
Let $1< p<\infty$, and let $q\in[1,\infty]$. Then:
\begin{enumerate}
\item[{\rm(i)}] if $q<p^\ast$, then $\ell_p$ has the $q$-Schur property;%is a sequentially $q$-angelic space;
\item[{\rm(ii)}] if $q\geq p^\ast$, then $\ell_p^0$ does not have the $q$-Schur property;
\end{enumerate}
\end{proposition}

Other results and non-defined notions used in the paper will be taken from published articles, \cite{Gab-Pel} and the classical books \cite{Jar} and \cite{NaB}.

%%%%%%%%%%%%%%%%%%%%%%%%%%%%%%%%%%%%%
%%%%%%%%%%%%%%%%%%%%%%%%%%%%%%%%%%%%%
%%%%%%%%%%%%%%%%%%%%%%%%%%%%%%%%%%%%%
%%%%%%%%%%%%%%%%%%%%%%%%%%%%%%%%%%%%%
%%%%%%%%%%%%%%%%%%%%%%%%%%%%%%%%%%%%%

\section{Weakly (sequentially) $p$-angelic locally convex spaces} \label{subsec:p-angelic}

%%%%%%%%%%%%%%%%%%%%%%%%%%%%%%%%%%%%%
%%%%%%%%%%%%%%%%%%%%%%%%%%%%%%%%%%%%%
%%%%%%%%%%%%%%%%%%%%%%%%%%%%%%%%%%%%%
%%%%%%%%%%%%%%%%%%%%%%%%%%%%%%%%%%%%%
%%%%%%%%%%%%%%%%%%%%%%%%%%%%%%%%%%%%%

Recall that a Tychonoff space $X$  is called {\em Fr\'{e}chet--Urysohn} if for any cluster point $a\in X$ of a subset $A\subseteq X$ there is a sequence $\{ a_n\}_{n\in\w}\subseteq A$ which converges to $a$. Recall also that a Tychonoff space $X$ is called an {\em angelic space} if (1) every relatively countably compact subset of $X$ is relatively compact, and (2) any compact subspace of $X$ is Fr\'{e}chet--Urysohn. A natural class of locally convex spaces which are angelic in the weak topology (=weakly angelic) is given in the next assertion proved in  Lemma 2.8 of \cite{Gab-Pel}.
\begin{lemma} \label{l:angelic-strict-LF}
Let $E$ be a locally convex space whose closed bounded sets are weakly angelic {\rm(}for example, $E$ is a strict $(LF)$-space{\rm)}. Then $E$ is weakly angelic.
\end{lemma}
Note that any subspace of an angelic space is angelic, and a subset $A$ of an angelic space $X$ is (relatively) compact if and only if it is (relatively) countably compact if and only if $A$ is (relatively) sequentially compact, see Lemma 0.3 of \cite{Pryce}. Being motivated by the last property we introduced in \cite{Gabr-free-resp} a new class of  Tychonoff spaces:  $X$ is called {\em sequentially angelic} if a subset $K$ of $X$ is compact if and only if $K$ is sequentially compact. The class of sequentially angelic spaces is strictly wider than the class of angelic spaces (indeed, if $X$ is a countably compact space without infinite compact subsets, then $X$ is trivially sequentially angelic but not angelic). Below we introduce a weak $p$-version of (sequentially) angelic locally convex spaces.

\begin{definition} \label{def:p-seq-angelic} {\em
Let $p\in[1,\infty]$. A locally convex space $E$ is called
\begin{enumerate}
\item[$\bullet$] a {\em weakly sequentially $p$-angelic space} if the family of all relatively weakly sequentially $p$-compact sets in $E$ coincides with the family of all relatively weakly compact subsets of $E$;
\item[$\bullet$] a {\em weakly $p$-angelic space}  if it is  a weakly sequentially $p$-angelic space and each weakly compact subset  of $E$ is Fr\'{e}chet--Urysohn. \qed
\end{enumerate}    }
\end{definition}

Although the product of two angelic spaces can be not angelic, the class of  weakly sequentially $p$-angelic spaces is countably productive as the following assertion shows.
\begin{proposition} \label{p:weakly-p-ang-productive}
Let $p\in[1,\infty]$, and let $\{E_i\}_{i\in I}$ be a non-empty family of locally convex spaces.
\begin{enumerate}
\item[{\rm(i)}]  $E=\bigoplus_{i\in I} E_i$ is weakly sequentially $p$-angelic if and only if so are all summands $E_i$.
\item[{\rm(ii)}] If $I=\w$, then $E=\prod_{i\in \w} E_i$  is weakly sequentially $p$-angelic if and only if so are all factors $E_i$.
\end{enumerate}
\end{proposition}
%The class of weakly sequentially $p$-angelic spaces is closed under taking arbitrary sums and countable products.
\begin{proof}
(i) Since relatively weakly compact subsets and relatively weakly sequentially $p$-compact subsets of $E$ are bounded, they sit in $\bigoplus_{i\in F} E_i$ for some finite subset $F$ of $I$. Therefore (i) follows from (ii).

(ii) Recall that $E_w=\prod_{i\in I} (E_i)_w$, see  Theorem 8.8.5 of \cite{Jar}.

Assume that $E$ is a  weakly sequentially $p$-angelic space. For every $i\in\w$, the space $E_i$ can be considered as a closed subspace of $E$. Let $i\in\w$. If $A$ is a relatively weakly compact subset of $E_i$, then $A$ is relatively weakly compact in $E$, and hence $A$ is a relatively weakly sequentially $p$-compact subset of $E$. Then, by Lemma \ref{l:image-p-seq-com}, the projection $A=\pi_i(A)$ of $A$ onto $E_i$ is relatively weakly sequentially $p$-compact. Conversely, if $A$ is a relatively weakly sequentially $p$-compact subset of $E_i$, Lemma \ref{l:image-p-seq-com} implies that $A$ is relatively weakly sequentially $p$-compact in $E$. Hence $A$ is  relatively weakly compact in $E$. Therefore $A$ is  relatively weakly compact in the closed subspace $(E_i)_w$ of $E_w$. Thus $E_i$ is a  weakly sequentially $p$-angelic space.

Assume now that all factors $E_i$ are weakly sequentially $p$-angelic. Let $A$ be a relatively weakly compact subset of $E$. Then for every $i\in\w$,  its projections $\pi_i(A)$ on $E_i$ is relatively weakly compact in $E_i$. Since $E_i$ is  weakly sequentially $p$-angelic, we obtain that $\pi_i(A)$ is relatively weakly sequentially $p$-compact. Hence, by Lemma \ref{l:seq-p-comp-prod}, the set $A$ is  relatively weakly sequentially $p$-compact. Conversely, if $A\subseteq E$ is  relatively weakly sequentially $p$-compact, then, by Lemma \ref{l:image-p-seq-com}, for every  $i\in\w$,  its projections $\pi_i(A)$ on $E_i$ is  relatively weakly sequentially $p$-compact. As $E_i$ is  weakly sequentially $p$-angelic, it follows that $\pi_i(A)$  is  relatively weakly compact in $E_i$. Therefore $\prod_{i\in\w} \pi_i(A)$ and hence also $A$ are relatively weakly compact in $E$. Thus $E$ is a weakly sequentially $p$-angelic space.\qed
\end{proof}

We select the following assertion.
\begin{proposition} \label{p:angelic-p-ang}
Every weakly angelic space $E$ is weakly $\infty$-angelic. Consequently, every strict $(LF)$-space being weakly angelic is a  weakly $\infty$-angelic space.
\end{proposition}

\begin{proof}
By definition $A$ is a relatively weakly sequentially $\infty$-compact subset of $E$ if and only if it is relatively  weakly sequentially compact, and hence, by the weak angelicity of $E$ and Lemma 0.3 of \cite{Pryce}, if and only if $A$ is relatively weakly compact. Therefore $E$ is a weakly sequentially $\infty$-angelic space. As $E$ is weakly angelic, each weakly compact subset  of $E$ is Fr\'{e}chet--Urysohn. Thus $E$ is a weakly $\infty$-angelic space.
The last assertion follows from Lemma \ref{l:angelic-strict-LF}.\qed
\end{proof}

The direct locally convex sum $\IF^{(\w)}$ of a countably-infinite family of the field $\IF$ is denoted by $\varphi$. It is well known that any bounded subset of $\varphi$ is finite-dimensional, see \cite{Jar}.

\begin{example} \label{exa:F-p-angelic}
The Fr\'{e}chet space $E=\IF^\w$ is weakly $p$-angelic for every $p\in[1,\infty]$.
\end{example}

\begin{proof}
Since $E$ is metrizable it is Fr\'{e}chet--Urysohn. As $E$ carries its weak topology and is complete, any relatively weakly sequentially $p$-compact set being precompact is relatively weakly compact. Conversely, let $A$ be a relatively weakly compact subset of $E$. To show that $A$ is relatively weakly sequentially $p$-compact, let $\{x_n\}_{n\in\w}$ be a sequence in $A$. Passing to a subsequence if needed we can assume that $x_n$ converges to some point $x\in E$. Now we proceed by induction. For $k=0$, choose $n_0\in\w$ such that $|x_n(0)-x(0)|< \tfrac{1}{2}$ for every $n\geq n_0$. Since $x_n$ converges to $x$ pointwise, one can find $n_k>n_{k-1}$ such that
\[
|x_n(j)-x(j)|< \tfrac{1}{2^{k+1}} \quad \mbox{ for every $n\geq n_k$ and each $j\leq k$}.
\]
Consider the subsequence $S=\{x_{n_k}\}_{k\in\w}$ of $\{x_n\}_{n\in\w}$. We prove that $S$ weakly $p$-converges to $x$ for $p=1$ and hence for every $p\in[1,\infty]$, which means that $A$ is  relatively weakly sequentially $p$-compact. By the choice of $S$, for every $i\in\w$ and each $k>i$, we have $|\langle e^\ast_i, x_{n_k}(i)-x(i)\rangle|<\tfrac{1}{2^{k+1}}$. Recall also that $E'=\varphi$. So, if $\chi=a_0 e^\ast_0 +\cdots+ a_j e^\ast_j\in E'$, then
\[
\sum_{k\in\w} |\langle\chi, x_{n_k}-x\rangle|\leq \sum_{i\leq j} \sum_{k\in\w} |a_i| \cdot|\langle e^\ast_i,  x_{n_k}(i)-x(i)\rangle| <\infty.
\]
Thus $S$ weakly $p$-converges to $x$.\qed
\end{proof}

Let $p\in[1,\infty]$. We shall say that a subset $A$ of an lcs $E$ is {\em weakly $p$-Fr\'{e}chet--Urysohn} if for any subset $B$ of $A$ and each $x\in \overline{B}^{\,\sigma(E,E')}$ there is a sequence $\{b_n\}_{n\in\w}$ in $B$ which weakly $p$-converges  to $x$.
The next lemma is an analogue of the aforementioned Lemma 0.3 of \cite{Pryce}.

%Since weakly sequentially $\infty$-compact sets are exactly weakly sequentially compact,
% The following lemma is clear.
%\begin{lemma} \label{l:angelic-p-angelic}
%{\rm(i)} An lcs $E$ is a weakly angelic space if and only if it is weakly $\infty$-angelic.

%{\rm(ii)} If $1\leq p<q\leq\infty$ and $E$ is a weakly $q$-angelic space, then $E$ is weakly $p$-angelic.
%\end{lemma}

\begin{lemma} \label{l:ws-p-angelic}
Let $p\in[1,\infty]$, and let $(E,\tau)$ be a weakly $p$-angelic space. Then:
\begin{enumerate}
\item[{\rm(i)}] a subset $A$ of $E$ is weakly  sequentially $p$-compact if and only if it is weakly compact, in this case $A$ is weakly $p$-Fr\'{e}chet--Urysohn;
\item[{\rm(ii)}] if $H$ is a linear subspace of $E$, then $H$ is weakly  $p$-angelic;
\item[{\rm(iii)}] if $\TTT$ is a locally convex vector topology on $E$ compatible with $\tau$, then the space $(E,\TTT)$ is a weakly  $p$-angelic space.
%\item[{\rm(iv)}] a subset $A$ of $E$ is relatively weakly sequentially $p$-compact if and only if it is relatively weakly compact;
\end{enumerate}
\end{lemma}

\begin{proof}
(i) Assume that $A$ is weakly  sequentially $p$-compact. Since $E$ is weakly sequentially $p$-angelic, $A$ is relatively weakly compact and hence the weak closure $\overline{A}^{\,w}$ of $A$ is a weakly compact subset of $E$. Therefore to show that $A$ is weakly compact we have to show that $\overline{A}^{\,w}=A$. Let $z\in\overline{A}^{\,w}$. By the definition of weak $p$-angelicity,  $\overline{A}^{\,w}$ is Fr\'{e}chet--Urysohn. Hence there is a sequence $\{x_n\}_{n\in\w}$ in $A$ which weakly converges to $z$. Since $A$ is weakly  sequentially $p$-compact, there exists a subsequence $\{x_{n_k}\}_{k\in\w}$ of  $\{x_n\}_{n\in\w}$  which weakly $p$-converges to a point $y\in A$. Taking into account that any weakly $p$-summable sequence is weakly null, it follows that $y=z$. Thus $\overline{A}^{\,w}=A$.

Conversely, assume that $A$ is a weakly compact subset of $E$. Since $E$ is weakly sequentially $p$-angelic it follows that $A$ is relatively weakly sequentially $p$-compact. Let $\{x_n\}_{n\in\w}$ be a sequence in $A$. Take a subsequence $\{x_{n_k}\}_{k\in\w}$ of  $\{x_n\}_{n\in\w}$  which weakly $p$-converges to a point $y\in E$. Since $A$ is weakly closed, we have $y\in A$. Thus $A$ is weakly  sequentially $p$-compact.

To show that $A$ is  weakly $p$-Fr\'{e}chet--Urysohn, let $B$ be a subset of $A$. Then $\overline{B}^{\,w}\subseteq A$. Let $z\in \overline{B}^{\,w}$. If $z\in B$, then the constant sequence $\{z\}$ weakly $p$-converges to $z$. If $z\not\in B$, then the Fr\'{e}chet--Urysohness of $A$ implies that there is a sequence $\{x_n\}_{n\in\w}$ in $B$ which weakly converges to $z$. Since $A$ is weakly  sequentially $p$-compact, there exists a subsequence $\{x_{n_k}\}_{k\in\w}$ of  $\{x_n\}_{n\in\w}$  which weakly $p$-converges to a point $y\in A$. It is clear that $y=z$. Thus $A$ is a weakly $p$-Fr\'{e}chet--Urysohn space.

(ii) Since $H_w$ is a subspace of $E_w$ (see Theorem 8.12.2 of \cite{NaB}), it suffices to show that $H$ is a weakly sequentially $p$-angelic space. Let $A$ be a relatively weakly sequentially $p$-compact subset of $H$. Since $H$ is a subspace of $E$, any weakly $p$-summable sequence in $H$ is also weakly $p$-summable in $E$. %Taking into account (vi) of Lemma \ref{l:prop-p-sum},
It follows that $A$ is a relatively  weakly sequentially $p$-compact subset of $E$. Since $E$ is weakly $p$-angelic, the weak closure $C:=\overline{A}^{\,\sigma(E,E')}$ of $A$ in the space $E$ is a weakly compact Fr\'{e}chet--Urysohn space. We show that $C\subseteq H$ which means that $A$ is a relatively weakly compact subset of $H$, as desired. To this end, let $z\in C$. Since $C$ is weakly Fr\'{e}chet--Urysohn, there is a sequence $\{x_n\}_{n\in\w}$ in $A$ which weakly converges (in $E$) to $z$. Since, by (i), $C$ is weakly  sequentially $p$-compact, there exists a subsequence $\{x_{n_k}\}_{k\in\w}$ of  $\{x_n\}_{n\in\w}$  which weakly $p$-converges to the point $z$. As $A$ is relatively weakly sequentially $p$-compact in $H$, it follows that $z\in H$. Thus $C\subseteq H$.

Conversely, assume that  $A$ is a relatively weakly compact subset of $H$. To show that $A$ is also relatively weakly sequentially $p$-compact, let $\{x_n\}_{n\in\w}$ be a sequence in $A$. Since the weak closure  $\overline{A}^{\,w}$ of $A$ in the space $H$ is weakly compact, the set  $\overline{A}^{\,w}$ is weakly compact also in $E$ and hence, by (i), $\overline{A}^{\,w}$ is weakly sequentially $p$-compact in $E$. Therefore there exists  a subsequence $\{x_{n_k}\}_{k\in\w}$ of $\{x_n\}_{n\in\w}$ which weakly $p$-converges (in $E$) to a point $z\in \overline{A}^{\,w}\subseteq H$. Hence, by the Hahn--Banach extension theorem, $\{x_{n_k}\}_{k\in\w}$ weakly $p$-converges to $z$ also in $H$. Thus $A$ is a relatively weakly sequentially $p$-compact subset of $H$.

(iii) follows from the fact that $(E,\TTT)_w=E_w$ (so that the property of being a weakly $p$-summable sequence and the property of being a (relatively) weakly compact set are the same for all compatible topologies).\qed
\end{proof}

\begin{proposition} \label{p:Lp-q-angelic}
Let $1< p<\infty$, and let $q\in[1,\infty]$. Then $\ell_p$ is weakly sequentially $q$-angelic if and only if it is weakly $q$-angelic if and only if $q\geq p^\ast$.
\end{proposition}

\begin{proof}
Since $\ell_p$ is reflexive and separable, every its weakly compact subset is even metrizable. Therefore $\ell_p$ is weakly sequentially $q$-angelic if and only if it is weakly $q$-angelic.

Assume that $\ell_p$ is a weakly sequentially $q$-angelic space. Consider two possible cases.

{\em Case 1. Assume that $q<p^{\ast}$. We show that $\ell_p$ is not weakly sequentially $q$-angelic.} Consider the standard unit basis $\{e_n\}_{n\in\w}$ in $\ell_p$. Then $e_n\to 0$ in the weak topology. Therefore to prove that $\ell_p$ is not weakly sequentially $q$-angelic, it suffices to show that $\{e_n\}_{n\in\w}$ does not have a weakly $q$-convergent subsequence.  To this end, taking into account that $\{e_n\}_{n\in\w}$ is weakly null, we show that for every its subsequence $\{e_{n_k}\}_{k\in\w}$ there is $\chi\in  \ell_{p^\ast}$ such that $\sum_{k\in\w} |\langle\chi,e_{n_k}\rangle|^q=\infty$. Take $\chi=(y_n)\in \ell_{p^\ast}\SM \ell_q$ such that $\supp(\chi)=\{n_k\}_{k\in\w}$. Then
\[
\sum_{k\in\w} |\langle\chi,e_{n_k}\rangle|^q=\sum_{k\in\w} |y_{n_k}|^q =\infty,
\]
as desired.
\smallskip

{\em Case 2. Assume that  $q\geq p^{\ast}$. We show that $\ell_p$ is weakly sequentially $q$-angelic.} If $K$ is a relatively  weakly sequentially $q$-compact subset of $\ell_p$, then it is  relatively weakly sequentially compact, and hence, by the Eberlein--\v{S}mulyan theorem, $K$ is  relatively weakly compact. Conversely, let $K$ be a  relatively weakly compact subset of $\ell_p$. Without loss of generality we can assume that $K=B_{\ell_p}$. Take an arbitrary sequence $\{x_n\}_{n\in\w}$ in $K$. Since, by the Eberlein--\v{S}mulyan theorem, $K$ is weakly sequentially compact, we can assume additionally that  $x_n$ weakly converges to some element $x\in K$. If $q=\infty$ we are done. Assume that $q<\infty$. To show that  $\{x_n\}_{n\in\w}$ has a weakly $q$-convergent subsequence we distinguish between two subcases.
\smallskip

{\em Subcase 2.1. The sequence  $\{x_n\}_{n\in\w}$ has a subsequence  $\{x_{n_k}\}_{k\in\w}$ which converges to $x$ in the norm topology.} Passing to a subsequence, we can assume that $\|x_{n_k}-x\|_{\ell_p} \leq \tfrac{1}{2^k}$. Then for every $\chi=(y_n)\in \ell_{p^\ast}$, we obtain %(we consider the case $q<\infty$, if $q=\infty$ the proof is even simpler)
\[
\sum_{k\in\w} |\langle\chi,x_{n_k}-x\rangle|^q\leq \|\chi\|_{\ell_{p^\ast}}^q \cdot \sum_{k\in\w} \|x_{n_k}-x\|_{\ell_p}^q \leq \|\chi\|_{\ell_{p^\ast}}^q \cdot \sum_{k\in\w} \tfrac{1}{2^{kq}} <\infty,
\]
which means that $\{x_{n_k}\}_{k\in\w}$ weakly $q$-converges to $x$, as desired.
\smallskip

{\em Subcase 2.2. The sequence  $\{x_n\}_{n\in\w}$ does note have a subsequence converging in the  norm topology.} Then there is $\e>0$ such that $\|x_{n}-x\|_{\ell_p}\geq \e$ for every $n\in\w$. By Proposition 2.1.3 of \cite{Al-Kal}, there are a basic subsequence $\{x_{n_k}-x\}_{k\in\w}$ of $\{x_{n}-x\}_{n\in\w}$ and a linear topological isomorphism $R: V:=\overline{\spn}\big( \{x_{n_k}-x\}_{k\in\w}\big) \to \ell_p$ such that
\begin{equation} \label{equ:Lp-weak-q-angelic-1}
R\big(x_{n_k}-x\big) =a_k e_k \quad \mbox{ for every $k\in\w$, where }\; a_k=\|x_{n_k}-x\|_{\ell_p},
\end{equation}
and such that $V$ is complemented in $\ell_p$. Note that
\[
\e\leq \|x_{n}-x\|_{\ell_p}\leq 2 \quad \mbox{ for every } \; k\in\w,
\]
and hence the map $Q:\ell_p\to \ell_p$ defined by $Q(\xi_k):= (a_k\xi_k)$, for $(\xi_k)\in \ell_p$, is a topological linear isomorphism. Therefore, replacing $R$ by $Q^{-1}\circ R$, we can assume that $a_k=1$ for every $k\in\w$.

Let $\chi\in (\ell_p)'=\ell_{p^\ast}$. Denote by $\eta$ the restriction of $\chi$ onto $V$, and let $\theta=(y_k)\in \ell_{p^\ast}$ be such that $R^\ast(\theta)=\eta$. Then (\ref{equ:Lp-weak-q-angelic-1}) and the inequality $q\geq p^{\ast}$ imply %(we assume that $q<\infty$ since the proof for the case $q=\infty$ is simpler)
\[
\begin{aligned}
\sum_{k\in\w} |\langle\chi,x_{n_k}-x\rangle|^q & =\sum_{k\in\w} |\langle\eta,x_{n_k}-x\rangle|^q= \sum_{k\in\w} |\langle R^\ast(\theta),x_{n_k}-x\rangle|^q \\
& =\sum_{k\in\w} |\langle \theta,R\big(x_{n_k}-x\big)\rangle|^q=\sum_{k\in\w} |y_k|^q <\infty.
\end{aligned}
\]
Therefore  $\{x_{n_k}\}_{k\in\w}$ weakly $q$-converges to $x$.

The Subcases 2.1 and 2.2 show that $\ell_p$ is a weakly sequentially $q$-angelic space. \qed
\end{proof}

One of the most important properties of angelic spaces is the following: if $(X,\tau)$ is an angelic space and $\TTT$ is a topology on $X$ finer than $\tau$, then also the space $(X,\TTT)$ is angelic. The next example shows that in  general weakly $p$-angelic spaces do not have an analogous property.

\begin{example} \label{exa:non-hereditary-p-angelic}
Let $1<p<\infty$ and $1\leq q< p^\ast$. Then  the Banach space $\ell_{p}$ is not weakly $q$-angelic, however, $\ell_{p}$ endowed with the pointwise topology $\tau_p$ induced from $\IF^\w$ is a weakly $q$-angelic space.
\end{example}

\begin{proof}
The space $\ell_{p}$ is not a weakly $q$-angelic by Proposition \ref{p:Lp-q-angelic}. On the other hand, the space $\big(\ell_{p},\tau_p\big)$ is a weakly $q$-angelic by Example \ref{exa:F-p-angelic} and (ii) of Lemma \ref{l:ws-p-angelic}.\qed
\end{proof}

%%%%%%%%%%%%%%%%%%%%%%%%%%%%%%%%%%%%%
%%%%%%%%%%%%%%%%%%%%%%%%%%%%%%%%%%%%%
%%%%%%%%%%%%%%%%%%%%%%%%%%%%%%%%%%%%%
%%%%%%%%%%%%%%%%%%%%%%%%%%%%%%%%%%%%%
%%%%%%%%%%%%%%%%%%%%%%%%%%%%%%%%%%%%%

\section{Characterizations of Dunford--Pettis type properties} \label{sec:char-DP}

%%%%%%%%%%%%%%%%%%%%%%%%%%%%%%%%%%%%%
%%%%%%%%%%%%%%%%%%%%%%%%%%%%%%%%%%%%%
%%%%%%%%%%%%%%%%%%%%%%%%%%%%%%%%%%%%%
%%%%%%%%%%%%%%%%%%%%%%%%%%%%%%%%%%%%%
%%%%%%%%%%%%%%%%%%%%%%%%%%%%%%%%%%%%%

%We start this section with the study of some basic properties the Grothendieck topology $\tau_{\Sigma'}$.
In the next lemma we summarize some properties of the Grothendieck topology $\tau_{\Sigma'}$ with respect to the original topology $\tau$ and the weak topology $\sigma(E,E')$. Recall that $\Sigma'(E')$ denotes the family of all absolutely convex, equicontinuous, weakly compact subsets of $E'_\beta$, and the Grothendieck topology $\tau_{\Sigma'}=\tau_{\Sigma'}(E)$ is the topology on $E$ of uniform convergence on the sets of $\Sigma'(E')$. Analogously, we denote by $\Sigma(E)$ the family of all absolutely convex and  weakly compact subsets of $E$.

Recall that a dense subspace $H$ of a locally convex space $E$ is called {\em large} if every bounded subset of $E$ is contained in the closure of a bounded subset of $H$ (for more information on large subspaces of $E$, see Chapter 8.3 of \cite{PB}).
\begin{lemma} \label{l:Sigma'}
Let $(E,\tau)$ be a locally convex space. Then:
\begin{enumerate}
\item[{\rm(i)}]  $\sigma(E,E')\subseteq \tau_{\Sigma'}\subseteq \tau$. In particular, $\tau_{\Sigma'}$ is compatible with $\tau$.
\item[{\rm(ii)}] $\sigma(E,E')=\tau_{\Sigma'}$ if and only if every $K\in\Sigma'(E')$ is finite-dimensional.
\item[{\rm(iii)}] $\tau_{\Sigma'}=\tau$ if and only if for every $U\in\Nn_0(E)$, the polar $U^\circ$ is a weakly compact subset of $E'_\beta$  if and only if for every $U\in\Nn_0(E)$, the polar $U^\circ$ is a weakly complete subset of $E'_\beta$.
\item[{\rm(iv)}] If $E$ is a semi-reflexive space, then $\tau_{\Sigma'}= \tau$.
\item[{\rm(v)}] If $E$ is a barrelled space and $H:=\big( E', \mu(E',E)\big)$, then $\tau_{\Sigma'}(H)=\mu(E',E)$.
\item[{\rm(vi)}] If $E$ is a normed space, then $\tau_{\Sigma'}=\tau$ if and only if the completion $\overline{E}$ of $E$ is a reflexive Banach space.
\item[{\rm(vii)}] If $H$ is a large subspace of $E$, then $(H,\tau_{\Sigma'}(H))$ is a subspace of $(E,\tau_{\Sigma'}(E))$.
\end{enumerate}
\end{lemma}

\begin{proof}
(i) The first inclusion $\sigma(E,E')\subseteq \tau_{\Sigma'}$ is clear. To show that $\tau_{\Sigma'}\subseteq \tau$, let $K\in \Sigma'(E')$. Then $K$ is equicontinuous and hence there is $U\in\Nn^c_0(E)$ such that $K\subseteq U^\circ$. Therefore $U=U^{\circ\circ}\subseteq K^\circ$. Thus $\tau_{\Sigma'}\subseteq \tau$.
\smallskip

(ii) Assume that   $\sigma(E,E')=\tau_{\Sigma'}$, and let $K\in\Sigma'(E')$. Then $K^\circ\in \tau_{\Sigma'}=\sigma(E,E')$, and hence there is a finite $F\subseteq E'$  such that $ F^\circ \subseteq K^\circ$. Therefore $K=K^{\circ\circ}$ is contained in $\spn(F)$. Thus $K$ is finite-dimensional.

Conversely, assume that every $K\in\Sigma'(E')$ is finite-dimensional. Then $K\subseteq \cacx(F)$ for some finite $F\subseteq E'$. Therefore $F^\circ \subseteq K^\circ$ and $K^\circ \in  \sigma(E,E')$, so $\tau_{\Sigma'}\subseteq \sigma(E,E')$. Thus, by (i), we obtain $\sigma(E,E')=\tau_{\Sigma'}$.
\smallskip

(iii) Assume that  $\tau_{\Sigma'}=\tau$, and let  $U\in\Nn_0(E)$. Without loss of generality we assume that $U=U^{\circ\circ}$. Choose $K\in\Sigma'(E')$ such that $K^\circ\subseteq U$. Then $U^\circ\subseteq K^{\circ\circ}=K$. Since, by the Alaoglu theorem, $U^\circ$ is weak$^\ast$ compact it follows that $U^\circ$ is a weakly closed subset of the weakly compact set $K$. Thus $U^\circ$ is weakly compact and hence weakly complete.

Assume now that for every $U\in\Nn_0(E)$, the polar $U^\circ$ is a weakly complete subset of $E'_\beta$. To show that $\tau_{\Sigma'}=\tau$, by (i), it suffices to prove that $\tau\subseteq \tau_{\Sigma'}$. Fix an arbitrary $U\in\Nn_0(E)$, we can assume that $U$ is absolutely convex and closed. Then, by Lemma \ref{l:polar-open}, $U^\circ$ is strongly bounded and hence weakly precompact in $E'_\beta$. Being also weakly complete, $U^\circ$ is weakly compact. Therefore $U^\circ\in \Sigma'(E')$ and hence $U=U^{\circ\circ}\in \tau_{\Sigma'}$.  Thus $\tau\subseteq \tau_{\Sigma'}$.
\smallskip

(iv) By (iii) it suffices to show that for every $U\in\Nn_0(E)$, the polar $U^\circ$ is a weakly compact subset of $E'_\beta$. Since $E$ is semi-reflexive, the weak topology of $E'_\beta$ is exactly $\sigma(E',E)$ and hence, by the Alaoglu theorem, $U^\circ$ is a weakly compact subset of $E'_\beta$. % By (i) it remains to prove the inverse inclusion $\tau \subseteq \tau_{\Sigma'}$. Let $U\in \tau$. Since $\tau=\mu(E,E')$ without loss of generality we can assume that $U=C^\circ$ for some absolutely convex $\sigma(E',E)$-compact subset $C$ of $E'$. Since $E$ is semi-reflexive, the strong topology $\beta(E',E)$ is compatible with $\sigma(E',E)$ and hence $C$ is an absolutely convex weakly compact subset of $E'_\beta$. As $U^\circ=C^{\circ\circ}=C$ we obtain that $C$ is also equicontinuous. Therefore $C\in \Sigma'(E')$ and hence $U=C^\circ\in \tau_{\Sigma'}$. Thus $\tau \subseteq \tau_{\Sigma'}$.
%
%(iii) "$E$ is a barrelled space and" Since $E$ is barrelled, we have $E=E_\beta$. As $\mu(E',E)$ is compatible with the duality $(E,E')$, it follows that $H'_\beta=E_\beta =E$ and hence $H''=H$, i.e. $H$ is semi-reflexive. Note that the space $H$ is Mackey by definition. Now (ii) applies.
\smallskip

(v) Observe that $H'_\beta =E_\beta$. Since $E$ is barrelled, we also have $E=E_\beta$. Therefore $H''=E'=H$ algebraically, and hence $H$ is semi-refelexive. Now (iv) applies.
\smallskip

(vi) Assume that $\tau_{\Sigma'}=\tau$. Then there is $K\in\Sigma'(E')$ such that $K^\circ\subseteq B_E$. Therefore $B^\circ_E \subseteq K^{\circ\circ}=K$ and hence the closed unit ball $B^\circ_E$ of the Banach space $E'_\beta$ is weakly compact. Thus, by Theorem 3.111 and Proposition 3.112 of \cite{fabian-10}, the Banach space $\overline{E}$ is reflexive.

Conversely, assume that the Banach space $\overline{E}$ is reflexive. Then, by Theorem 3.111 and Proposition 3.112 of \cite{fabian-10}, $B^\circ_E =B^\circ_{\overline{E}}$ is weakly compact. Therefore $B^\circ_E\in \Sigma(E'_\beta)$. Since $B^\circ_E$ is equicontinuous we obtain $B^\circ_E\in \Sigma'(E')$ and hence $B_E= B_E^{\circ\circ}\in \tau_{\Sigma'}$, therefore $\tau\subseteq \tau_{\Sigma'}$. By (i), $\tau_{\Sigma'}\subseteq \tau$. Thus $\tau_{\Sigma'}=\tau$.
\smallskip

(vii) Taking into account that $H'=E'$ algebraically, to show that $(H,\tau_{\Sigma'}(H))$ is a subspace of $(E,\tau_{\Sigma'}(E))$ it suffices to prove that $\Sigma'(H')=\Sigma'(E')$. Recall that, by Proposition 2.4 of \cite{Gab-Pel}, %\ref{p:large-charac}
we have $E'_\beta=H'_\beta$.

Let $K\in\Sigma'(E')$, so $K$ is an absolutely convex and weakly compact subset of $H'_\beta=E'_\beta$. Take $U\in\Nn_0^c(E)$ such that $K\subseteq U^\circ$. Set $V:= U\cap H$. Then $V\in\Nn_0^c(H)$ and $V$ is dense in $U$. Therefore $V^\circ=U^\circ$ and hence $K\subseteq V^\circ$. Thus $K$ is  equicontinuous and hence $K\in \Sigma'(H')$.

Conversely, let $K\in \Sigma'(H')$. Then $K$ is an absolutely convex and weakly compact subset of $E'_\beta=H'_\beta$. Take $V\in\Nn_0^c(H)$ such that $K\subseteq V^\circ$. Set $U:= \overline{V}^E$. Then $U\in\Nn_0^c(E)$, and the density of $H$ in $E$ implies $U^\circ=V^\circ$. Therefore $K\subseteq U^\circ$ is equicontinuous. Thus $K\in \Sigma'(E')$.\qed
\end{proof}

The equivalence of the conditions (i)-(iii) in the next theorem are exactly the conditions (DP$_1$)-(DP$_3$) from \cite[p.633]{Edwards} which, following Grothendieck  \cite{Grothen}, define the $DP$ property in locally convex spaces. For reader's convenience and for the sake of completeness of the article we recall the proof of (i)-(vii).

\begin{theorem} \label{t:def-DP}
For a locally convex space $(E,\tau)$  the following conditions are equivalent:
\begin{enumerate}
\item[{\rm(i)}] if $T$ is an operator from $E$ into a locally convex space $L$ that transforms bounded sets into relatively weakly compact sets, then $T(A)$  is precompact in $L$ for every $A\in\Sigma(E)$;
\item[{\rm(ii)}] as in {\rm(i)}, but $L$ is assumed to be a Banach space; that is, $E$ has the $DP$ property;
\item[{\rm(iii)}] each $A\in\Sigma(E)$ is precompact for the Grothendieck topology $\tau_{\Sigma'}$;
\item[{\rm(iv)}] each $B\in \Sigma'(E')$ is precompact for $\mu(E',E)$;
\item[{\rm(v)}] if $T$ is an operator from $E$ into a locally convex space $L$ that transforms bounded sets into relatively weakly compact sets, then $T(A)$  is compact in $L$ for every $A\in\Sigma(E)$;
\item[{\rm(vi)}] as in {\rm(v)}, but $L$ is assumed to be a Banach space;
\item[{\rm(vii)}] each $A\in\Sigma(E)$ is  compact for the  Grothendieck topology $\tau_{\Sigma'}$;
\item[{\rm(viii)}] the space $\big(E,\tau_{\Sigma'}\big)$ has the weak Glicksberg property.
\end{enumerate}
\end{theorem}

\begin{proof}
The conditions (i)-(iv) are equivalent by Theorem 9.3.4 of \cite{Edwards} applied to $\mathfrak{S}=\Sigma(E)$ %Theorem \ref{t:G-dual-top-3}
and taking into account that, by the Mackey--Arens theorem, the topology $\TTT_{\Sigma}$ on $E'$ is the Mackey topology $\mu(E',E)$. The conditions (v)-(vii) are equivalent by Remark after the proof of Theorem 9.3.4 of  \cite{Edwards} %Theorem \ref{t:G-dual-top-3}
because the family $\Sigma(E)$ contains weakly compact sets.

(ii)$\Ra$(vi) Let $A\in\Sigma(E)$. By (ii), $T(A)$ is precompact in the Banach space $L$. On the other hand, since $A$ is weakly compact, $T(A)$ is weakly compact in $L$. Therefore $T(A)$  is a closed subset of $L$. As $L$ is complete, it follows that $T(A)$ is a compact subset of $L$.

(vi)$\Ra$(ii) is trivial.

(vii)$\Ra$(viii) Let $A$ be a weakly compact, absolutely convex subset of  $\big(E,\tau_{\Sigma'}\big)$. Since, by Lemma \ref{l:Sigma'}, $\tau_{\Sigma'}$ is compatible with the topology $\tau$ of $E$, the set $A$ belongs to $\Sigma(E)$. Therefore, by (vii), $A$ is compact for the topology $\tau_{\Sigma'}$. Thus $\big(E,\tau_{\Sigma'}\big)$ has the weak Glicksberg property.

(viii)$\Ra$(vii) Let $A\in \Sigma(E)$. Since, by Lemma \ref{l:Sigma'}, $\tau_{\Sigma'}$ and $\tau$ are compatible, $A$ is a weakly compact, absolutely convex subset of  $\big(E,\tau_{\Sigma'}\big)$. By  the weak Glicksberg property, we obtain that $A$ is a compact subset of $E$ for $\tau_{\Sigma'}$.\qed
\end{proof}

%Being motivated by the Krein theorem \cite[\S~24.5(4)]{Kothe} (which states that if $K$ is a weakly compact subset of an lcs $E$, then $\overline{\mathrm{acx}}(K)$ is weakly compact if and only if $\overline{\mathrm{acx}}(K)$ is $\mu(E,E')$-complete) and following \cite{Gabr-free-resp}, a locally convex space $E$ has the {\em Krein property} if the closed absolutely convex hull of each weakly compact subset $K$ of $E$ is weakly compact.

It is well known that a Banach space $X$ has the $DP$ property if and only if every $\sigma(X',X'')$-compact subset of  $X'$ is $\mu(X',X)$-compact. Below we generalize this fact.

\begin{corollary} \label{c:DP-quasi}
A quasibarrelled space $E$ has the $DP$ property if and only if every $\sigma(E',E'')$-compact subset of $E'$ is $\mu(E',E)$-compact.
\end{corollary}

\begin{proof}
The sufficiency immediately follows from (iv) of Theorem \ref{t:def-DP}. To prove the necessity let $K$ be a $\sigma(E',E'')$-compact subset of $E'$.
Since $E$ is quasibarrelled, the space $(E'_\beta)_w$ is quasi-complete by Proposition 11.2.4 of \cite{Jar}, and hence $E'_\beta$ has the Krein property.
Therefore the absolutely convex closed hull $C$ of $K$ is a weakly compact subset of $E'_\beta$. Quasibarrelledness of $E$ implies that $C$ is equicontinuous, i.e., $C\in \Sigma'(E')$. Therefore, by  (iv) of Theorem \ref{t:def-DP}, $C$ is $\mu(E',E)$-precompact. As the weak topology of $E'_\beta$ is finer than $\sigma(E',E)$, it follows that $C$ and hence also $K$ are $\sigma(E',E)$-compact, and so $K$ is $\mu(E',E)$-complete by Theorem 3.2.4 of \cite{Jar}. Being complete and precompact $K$ is  $\mu(E',E)$-compact.\qed
\end{proof}

In \cite{Pelcz-62}  Pe{\l}czy\'{n}ski introduced and studied the property $V$ and the property $V^\ast$ for Banach spaces. These properties were generalized  by Chen, Ch\'{a}vez-Dom\'{\i}nguez, and Li in \cite{LCCD} for every $p\in[1,\infty]$. Being motivated by these articles we defined and study in \cite{Gab-Pel} the  Pe{\l}czy\'{n}ski's properties  $V_p$ and $V^\ast_p$ for locally convex spaces as follows. Let $1\leq p\leq q\leq\infty$, and let $E$ be a locally convex space. Then:
\begin{enumerate}
\item[$\bullet$] a subset $A$ of $E$ is called a {\em $(p,q)$-$(V^\ast)$ set} if
\[
\Big(\sup_{a\in A} |\langle \chi_n, a\rangle|\Big)\in \ell_q \; \mbox{ if $q<\infty$, } \; \mbox{ or }\;\; \Big(\sup_{a\in A} |\langle \chi_n, a\rangle|\Big)\in c_0 \; \mbox{ if $q=\infty$},
\]
for every weakly $p$-summable sequence $\{\chi_n\}_{n\in\w}$ in $E'_\beta$, $(p,\infty)$-$(V^\ast)$ sets are called {\em $p$-$(V^\ast)$ sets};
\item[$\bullet$] a subset  $B$ of $E'$ is called a {\em $p$-$(V)$ set} if
$
\lim_{n\to\infty} \sup_{\chi\in B} |\langle\chi,x_n\rangle|=0\;
$
 for every weakly $p$-summable sequence $\{x_n\}_{n\in\w}$ in $E$,
\end{enumerate}
and the space $E$ is said to have
\begin{enumerate}
\item[$\bullet$] the {\em property $V_p^\ast$} if every $p$-$(V^\ast)$ set in $E$ is relatively weakly compact,
\item[$\bullet$] the {\em property $V_p$} if every $p$-$(V)$ set in $E'_\beta$ is relatively weakly compact.
\end{enumerate}

\begin{lemma} \label{l:null-seq}
Let $\tau$ and $\TTT$ be two locally convex vector topologies on a vector space $E$ such that $\tau\subseteq \TTT$. If $S=\{x_n\}_{n\in\w}$ is $\tau$-null and $\TTT$-precompact, then $S$ is $\TTT$-null. Consequently, if $S$ is weakly $\TTT$-null and $\TTT$-precompact, then $S$ is $\TTT$-null.
%Let $S=\{x_n\}_{n\in\w}$ be a weakly null sequence in a locally convex space $E$. If $S$ is precompact in $E$, then $x_n\to 0$ in $E$.
\end{lemma}

\begin{proof}
Let $\overline{E}$ be the completion of $(E,\TTT)$. Since $S$ is $\TTT$-precompact, its completion $\overline{S}$ in $\overline{E}$ is compact. By assumption $S$ is $\tau$-null, so $S\cup \{0\}$ is $\tau$-compact. Since $\tau\subseteq \TTT$ we obtain $\overline{S}=S\cup \{0\}$. As the identity map $I:(E,\TTT)\to (E,\tau)$ is continuous, $I{\restriction}_{\overline{S}}$ is a homeomorphism. Since $S$ is $\tau$-null, it follows that  $x_n\to 0$ in $(E,\TTT)$.\qed
\end{proof}

If $p=\infty$, the equivalence of the conditions (i), (i') and (v) in the next theorem motivate Grothendieck \cite{Grothen} to introduce the strict $DP$ property for locally convex spaces (they are exactly the conditions (SDP$_1$)-(SDP$_3$) from \cite[p.633]{Edwards}). %We provide a detailed proof of (i)-(v) for the sake of completeness and reader's convenience.

\begin{theorem} \label{t:strict-DPp}
Let $p\in[1,\infty]$, and let $E$ be a locally convex space. Then the following assertions are equivalent:
\begin{enumerate}
\item[{\rm(i)}] for every Banach space $L$, each operator $T\in\LL(E,L)$, which transforms bounded sets into relatively weakly compact sets,  is $p$-convergent; that is, $E$ has the strict $DP_p$ property;
\item[{\rm(i')}]  for every locally convex space $L$, each operator $T\in\LL(E,L)$, which transforms bounded sets into relatively weakly compact sets,  is $p$-convergent;
\item[{\rm(ii)}] for every locally convex  space $L$, each operator $T\in\LL(E,L)$, which transforms bounded sets into relatively weakly compact sets,  transforms each weakly sequentially $p$-precompact subset of $E$ into a precompact subset of $L$;
\item[{\rm(ii')}] as in {\rm(ii)}, but $L$ is assumed to be a Banach space;
\item[{\rm(iii)}] for  every  locally convex  space $L$, each operator $T\in\LL(E,L)$, which transforms bounded sets into relatively weakly compact sets,  transforms each weakly sequentially $p$-compact subset of $E$ into a precompact subset of $L$;
\item[{\rm(iii')}] as in {\rm(iii)}, but $L$ is assumed to be a Banach space;
\item[{\rm(iv)}] each weakly $p$-summable sequence in $E$ is precompact for the Grothendieck topology $\tau_{\Sigma'}$;
\item[{\rm(v)}] each $B\in \Sigma'(E')$ is precompact for $\TTT_\BB$, where $\BB$ is the family of all weakly $p$-summable sequences in $E$;
\item[{\rm(vi)}] the space $\big(E,\tau_{\Sigma'}\big)$ has the $p$-Schur property;
\item[{\rm(vii)}] for every Banach space $L$ and for each operator $T\in\LL(E,L)$, which transforms bounded sets into relatively weakly compact sets, $T^\ast(B)$ is a $p$-$(V)$ set for any bounded subset $B$ of $L'_\beta$.
\end{enumerate}
\end{theorem}

\begin{proof}
The equivalences (i)$\Leftrightarrow$(i')$\Leftrightarrow$(v)  (resp., (ii)$\Leftrightarrow$(ii') or (iii)$\Leftrightarrow$(iii')) immediately follow from Lemma \ref{l:null-seq} and
%the equivalence (i)$\Leftrightarrow$(ii) in Theorem \ref{t:G-dual-top-3}
Theorem 9.3.4 of \cite{Edwards}
applied to the family $\mathfrak{S}$ of all weakly $p$-summable sequences (resp., to the family $\mathfrak{S}$ of all weakly sequentially $p$-precompact sets or to the family $\mathfrak{S}$ of all weakly sequentially $p$-compact sets) in $E$.

(i)$\Ra$(iv) Let $S=\{x_n\}_{n\in\w}$ be a weakly $p$-summable sequence in $E$. Then %, by  the equivalence (i)$\Leftrightarrow$(iii) in Theorem \ref{t:G-dual-top-3} applied to the family $\BB$ of all weakly $p$-summable sequences,
 Theorem 9.3.4 of \cite{Edwards}, applied to the family  of all weakly $p$-summable sequences, implies that
$S$ is precompact for $\tau_{\Sigma'}$.

(iv)$\Ra$(i) Let $L$ be a Banach space, and let $T\in\LL(E,L)$ transform bounded sets into relatively weakly compact sets. If  $S=\{x_n\}_{n\in\w}$ is a weakly $p$-summable sequence in $E$,
%the equivalence (i)$\Leftrightarrow$(iii) in Theorem \ref{t:G-dual-top-3}, applied to the family $\BB$ of all weakly $p$-summable sequences (which are weakly compact),
Theorem 9.3.4 of \cite{Edwards}, applied to the family  of all weakly $p$-summable sequences,
implies that $T(S)$ is precompact in  $L$. Since $T(S)$ is also weakly $p$-summable, Lemma \ref{l:null-seq} implies $T(x_n)\to 0$ in $L$. Thus $T$ is $p$-convergent.

%The equivalences (i)$\Leftrightarrow$(iv)$\Leftrightarrow$(v) follow from Theorem \ref{t:G-dual-top-3} applied to the family $\BB$ of all weakly $p$-summable sequences.
\smallskip

(i)$\Ra$(ii') Let $T:E\to L$  be an operator from $E$ into a Banach space $L$ which  transforms bounded sets into relatively weakly compact sets. Then, by (i), $T$ is $p$-convergent. Let $A$ be a weakly sequentially $p$-precompact subset of $E$. Then, by Proposition \ref{p:p-convergent-s}, $T(A)$ is  relatively sequentially compact in $L$. As $L$ is a Banach space, it follows that $T(A)$ is relatively compact hence precompact in $L$, as desired.
\smallskip

(ii')$\Ra$(iii') is trivial.
\smallskip

(iii')$\Ra$(i) Let $T:E\to L$  be an operator from $E$ into a Banach space $L$ which  transforms bounded sets into relatively weakly compact sets. Fix a weakly $p$-summable sequence $S=\{x_n\}_{n\in\w}$ in $E$. Since $S$ is weakly sequentially $p$-compact, by (iii'), the set $T(S)=\{T(x_n)\}_{n\in\w}$ is precompact in $L$.
Since $T(S)$ is also weakly $p$-summable, Lemma \ref{l:null-seq} implies $T(x_n)\to 0$ in $L$.
Thus $T$ is $p$-convergent.
\smallskip

(iv)$\Ra$(vi) Let $S=\{x_n\}_{n\in\w}$ be a weakly $p$-summable sequence in $\big(E,\tau_{\Sigma'}\big)$. Since, by (i) of Lemma \ref{l:Sigma'}, $\tau_{\Sigma'}$ is compatible with $\tau$, it follows that $S$ is weakly $p$-summable in $E$. Therefore, by (iv), the sequence $S$ is precompact in $\big(E,\tau_{\Sigma'}\big)$. Now Lemma \ref{l:null-seq} implies that $x_n\to 0$ in $\big(E,\tau_{\Sigma'}\big)$. Thus $\big(E,\tau_{\Sigma'}\big)$ has the $p$-Schur property.

(vi)$\Ra$(iv) Let $S=\{x_n\}_{n\in\w}$ be a weakly $p$-summable sequence in $E$. Since, by (i) of Lemma \ref{l:Sigma'}, $\tau_{\Sigma'}$ is compatible with $\tau$, it follows that $S$ is weakly $p$-summable in $\big(E,\tau_{\Sigma'}\big)$. As $\big(E,\tau_{\Sigma'}\big)$ has the $p$-Schur property, we obtain that $x_n \to 0$ in $\tau_{\Sigma'}$. In particular, $S$ is  precompact for $\tau_{\Sigma'}$.
\smallskip

The equivalence (i)$\Leftrightarrow$(vii) follows from Theorem 14.6 of \cite{Gab-Pel}  %{t:p-convergent-2}
which states that an operator $T:E\to L$ is $p$-convergent if and only if $T^\ast(B)$ is a $p$-$(V)$ set in $E'$ for every bounded subset $B$ of the Banach space $L'_\beta$.\qed
\end{proof}

\begin{corollary} \label{c:weak-Glick-DP}%{l:weak-Glick}
Let $p\in[1,\infty]$, and let $(E,\tau)$ be a locally convex space. Then:
\begin{enumerate}
\item[{\rm(i)}] if $E$ has the weak Glicksberg property, then it has the $DP$ property;
\item[{\rm(ii)}] if $E$ has the $p$-Schur property, then it has the strict $DP_p$ property;
\item[{\rm(iii)}] if $E$ is semi-Montel, then it has the $DP$-property and the strict $DP_p$-property.
\end{enumerate}
\end{corollary}

\begin{proof}
(i), (ii): By (i) of Lemma \ref{l:Sigma'}, we have  $\sigma(E,E')\subseteq \tau_{\Sigma'}\subseteq \tau$. It follows that also $(E,\tau_{\Sigma'})$ has the weak Glicksberg property or  the $p$-Schur property, respectively. Thus $E$ has the $DP$ property or the strict $DP_p$ property by Theorem \ref{t:def-DP} or Theorem \ref{t:strict-DPp}, respectively.

(iii) Assume that $E$ is semi-Montel. Then, by definition, $E$ has the Glicksberg property. Therefore, by (i) and (ii) and Lemma \ref{l:weak-Glick}, $E$  has the $DP$-property and the strict $DP_p$-property.\qed%If $E$ is semi-reflexive, then $\tau_{\Sigma'}=\tau$ by (iv) of Lemma \ref{l:Sigma'}.\qed\qed
%(i) Let $T:E\to L$ be an operator into a Banach space $L$ which transforms bounded sets into relatively weakly compact sets. Fix an arbitrary absolutely convex weakly compact set $K$ of $E$. Since $E$ has the weak Glicksberg property, $K$ is a compact subset of $E$. Therefore $T(K)$ is a compact subset of $L$. Thus $E$ has the $DP$ property.
%
%(ii) Let $T:E\to L$ be an operator into a Banach space $L$ which transforms bounded sets into relatively weakly compact sets. Fix an arbitrary weakly $p$-summable sequence $S$ in $E$. Since $E$ has the $p$-Schur property, $S$ is $\tau$-null. Therefore $T(S)$ is a null sequence in $L$. Thus $E$ has the strict $DP_p$ property.\qed
\end{proof}

In the case when $\tau\subseteq \tau_{\Sigma'}$ we can reverse (i) and (ii) in Corollary \ref{c:weak-Glick-DP}.
\begin{corollary} \label{c:semi-reflexive-strict-DP}
Let $p\in[1,\infty]$, and let $(E,\tau)$ be a locally convex spaces such that $\tau=\tau_{\Sigma'}$. Then:
\begin{enumerate}
\item[{\rm(i)}] $E$ has the $DP$-property if and only if it has the weak Glicksberg property;
\item[{\rm(ii)}]  $E$ has the strict $DP_p$-property if and only if it has the $p$-Schur property.
\end{enumerate}
\end{corollary}

\begin{proof}
The clause (i) follows from the equivalence (ii)$\Leftrightarrow$(viii) of Theorem \ref{t:def-DP}, and (ii) follows from the equivalence (i)$\Leftrightarrow$(vi) of Theorem \ref{t:strict-DPp}.\qed
\end{proof}

Let $(E,\tau)$ be a locally convex space, and let $q\in[1,\infty]$. We denote by $\tau_{S_q}$ the polar topology on $E$ of uniform convergence on weakly $q$-summable sequences in $E'_\beta$. Below we list some basic properties of the ``sequentially-open'' topology $\tau_{S_q}$ which will be used repeatedly in what follows.

\begin{lemma} \label{l:Sigma-seq}
Let $q\in[1,\infty]$, and let $(E,\tau)$ be a locally convex space. Then:
\begin{enumerate}
\item[{\rm(i)}]  $\sigma(E,E')\subseteq \tau_{S_q}$.
\item[{\rm(ii)}] $E$ is a $q$-quasibarrelled space if and only if $\tau_{S_q}\subseteq \tau$.
\item[{\rm(iii)}] $\tau_{S_q}\subseteq \mu(E,E')$ if and only if for every weakly $q$-summable sequence $S=(\chi_n)$ in $E'_\beta$, the absolutely convex hull of $S$ is relatively weak$^\ast$ compact.
\item[{\rm(iv)}] If $E'_\beta$ is locally complete (for example, $E$ is quasibarrelled), then $\tau_{S_q}$ is compatible with $\tau$.
\item[{\rm(v)}] $\tau_{S_q}\subseteq \tau_{\Sigma'}$ if and only if $E$ is a $q$-quasibarrelled space whose strong dual $E'_\beta$ is  locally complete {\rm(}for example, $E$ is quasibarrelled{\rm)}.
%\item[{\rm(iv)}] If $L$ is a locally convex space and $T\in\LL(E,L)$, then $T\in\LL\big((E,\tau_{\Sigma'}),L\big)$.
\end{enumerate}
\end{lemma}

\begin{proof}
(i) Let $\{\chi_0,\dots,\chi_k\}$ be a finite subset of $E'$. For every $n>k$, set $\chi_n=0$. It is clear that $\{\chi_n\}_{n\in\w}$ is a weakly $q$-summable sequence in $E'_\beta$. Now it is evident that $\sigma(E,E')\subseteq \tau_{S_q}$.

(ii) Assume that $E$ is a $q$-quasibarrelled space. Let $S=\{\chi_n\}_{n\in\w}$ be a weakly $q$-summable sequence in $E'_\beta$. Since $E$ is $q$-quasibarrelled, $S$ is equicontinuous. Take a closed absolutely convex neighborhood $U$ of zero in $E$ such that $S\subseteq U^\circ$. Then $U^{\circ\circ}=U\subseteq S^\circ$. Thus $\tau_{S_q}\subseteq \tau$.

Conversely, assume that $\tau_{S_q}\subseteq \tau$. To show that $E$ is $q$-quasibarrelled, let $S=\{\chi_n\}_{n\in\w}$ be a weakly $q$-summable sequence in $E'_\beta$. Then the inclusion $\tau_{S_q}\subseteq \tau$ implies that there is $U\in \Nn_0(E)$ such that $U\subseteq S^\circ$. Then $S\subseteq U^\circ$, and hence the sequence $S$ is equicontinuous. Thus  $E$ is $q$-quasibarrelled.

(iii) By the Mackey--Arens theorem, the inclusion $\tau_{S_q}\subseteq \mu(E,E')$ holds if and only if for every weakly $q$-summable sequence $S=(\chi_n)$ in $E'_\beta$, there is a weak$^\ast$ compact, absolutely convex subset $K$ of $E'$ such that $K^\circ \subseteq S^\circ$ if and only if $S^{\circ\circ}\subseteq K^{\circ\circ}=K$, as desired.

%If $q\in[1,\infty]$,  $\tau_{S_q}\subseteq \mu(E,E')$ if and only if for every weakly $q$-summable sequence $S=(\chi_n)$ in $E'_\beta$, the absolutely convex hull of $S$ is relatively weak$^\ast$ compact. Therefore the closed absolutely convex hull $K:=\cacx(S)$ of $S$ is also equicontinuous, and, by the Alaoglu theorem, $K$ is weak$^\ast$ compact. Therefore, by the Mackey--Arens theorem, $K^\circ$ is a $\mu(E,E')$-neighborhood of zero. Thus $\tau_{S_q}\subseteq \mu(E,E')$.
\smallskip

(iv) Assume that $E$ is such that $E'_\beta$ is locally complete (if $E$ is quasibarrelled, then by Proposition 11.2.3 of \cite{Jar}, $E'_\beta$ is quasi-complete and hence locally complete).
Then the weak closure of the absolutely convex hull $K$ of any weakly $q$-summable sequence $(\chi_n)$ in $E'_\beta$ is  weakly compact, and hence $K$ is also a weak$^\ast$ compact, absolutely convex subset of $E'$. Now the Mackey--Arens theorem implies that $\tau_{S_q}$ is weaker than the Mackey topology $\mu(E,E')$. %As weakly $q$-summable sequences trivially swallow $\FF(E')$, we obtain $\sigma(E,E')\subseteq \tau_{S_q}$.
This fact and (i) imply that $\tau_{S_q}$ is compatible with $\tau$.
\smallskip

(v) Assume that $\tau_{S_q}\subseteq \tau_{\Sigma'}$. Let  $S=(\chi_n)$ be a weakly $q$-summable sequence in $E'_\beta$. Then there is $K\in \Sigma'(E')$ such that $K^\circ\subseteq S^\circ$ and hence $S\subseteq S^{\circ\circ}\subseteq K$. Since $K$ is equicontinuous, it follows that also $S$ is equicontinuous and hence $E$ is $q$-quasibarrelled. On the other hand, since $K$ is weakly compact and absolutely convex in $E'_\beta$, we obtain that the closed, absolutely convex hull of $S$ is a weakly compact subset of $E'_\beta$. Therefore, by Theorem \ref{t:weakly-p-lc}, $E'_\beta$ is locally complete.

Conversely, assume that $E$ is a $q$-quasibarrelled space whose strong dual $E'_\beta$ is locally complete. If $S=(\chi_n)$ is a weakly $q$-summable sequence in $E'_\beta$, then the local completeness of $E'_\beta$ implies that the closed absolutely convex hull $K:=\cacx(S)$ is a weakly compact subset of $E'_\beta$. Since $E$ is $q$-quasibarrelled, $S$ and hence also $K$ are equicontinuous. Therefore $K\in\Sigma'(E')$. It follows that $\tau_{S_q}\subseteq \tau_{\Sigma'}$, as desired.

%{t:weakly-p-lc}

If $E$ is quasibarrelled, then it is trivially $q$-quasibarrelled. By Proposition 11.2.3 of \cite{Jar}, the space $E'_\beta$ is quasi-complete and hence locally complete.\qed
\end{proof}

Below we characterize locally convex spaces with the sequential $DP_{(p,q)}$ property.

\begin{theorem} \label{t:seq-DPpq}
Let $p,q\in[1,\infty]$, $(E,\tau)$ be a locally convex space, $\BB_E=\ell_p^w(E)$ $($or $c_0^w(E)$ if $p=\infty$$)$, $\BB_{E'}=\ell_q^w(E'_\beta)$ $($or $c_0^w(E'_\beta)$ if $q=\infty$$)$, and let $\Id:E\to E$ be the identity map. Then the following assertions are equivalent:
\begin{enumerate}
\item[{\rm(i)}] for each $A=(x_n)\in \BB_E$, the set $A$ is precompact {\rm(}resp., compact{\rm)} for $\TTT_{\BB_{E'}}$;
\item[{\rm(ii)}] for each $B=(\chi_n)\in \BB_{E'}$, the set $B$ is precompact  {\rm(}resp., compact{\rm)} for $\TTT_{\BB_{E}}$;
\item[{\rm(iii)}] for each $A\in \BB_E$, $\Id{\restriction}_A: \big(A,\sigma(E,E'){\restriction}_A\big)\to \big(E,\TTT_{\BB_{E'}}\big)$ is uniformly continuous;
\item[{\rm(iv)}] for each $B\in \BB_{E'}$, $\Id^\ast{\restriction}_B: \big(B,\sigma(E',E){\restriction}_B\big)\to \big(E',\TTT_{\BB_{E}}\big)$ is uniformly continuous;
\item[{\rm(v)}] for each $A=(x_n)\in \BB_E$ and every $B=(\chi_n)\in \BB_{E'}$, the restriction to $B\times A$ of $\langle \eta,x\rangle$ is uniformly continuous for the product topology $\beta(E',E){\restriction}_B\times\sigma(E,E'){\restriction}_A$  $($or vice versa for $\sigma(E',E){\restriction}_B\times\beta(E,E'){\restriction}_A$$)$;
\item[{\rm(vi)}] for each $A=(x_n)\in \BB_E$ and every $B=(\chi_n)\in \BB_{E'}$, the restriction to $B\times A$ of $\langle \eta, x\rangle$ is uniformly continuous for the product topology $\sigma(E',E){\restriction}_B\times\sigma(E,E'){\restriction}_A$;
\item[{\rm(vii)}] $E$ has   the sequential $DP_{(p,q)}$ property;
\item[{\rm(viii)}] each relatively weakly sequentially $p$-compact set in $E$ is a $q$-$(V^\ast)$ set.
\end{enumerate}
Moreover, if $\TTT_{\BB_{E'}}$ is compatible with $\tau$ {\rm(}for example, $E$ is such that $E'_\beta$ is locally complete or $E$ is quasibarrelled{\rm)}, then {\rm(i)-(viii)} are equivalent to
\begin{enumerate}
\item[{\rm(ix)}] $\big(E,\TTT_{\BB_{E'}}\big)$ has   the $p$-Schur property.
\end{enumerate}
\end{theorem}

\begin{proof}
Set $L:=E$ and define $T:E\to L$ by $T:=\Id$.

By (i) of Lemma \ref{l:Sigma-seq}, we have $\sigma(E,E')\subseteq \TTT_{\BB_{E'}}$. Hence, by Lemma \ref{l:null-seq}, if $A\in \BB_E$ is $\TTT_{\BB_{E'}}$-precompact, then $A\cup\{0\}$ is $\TTT_{\BB_{E'}}$-compact. Analogously, it is clear that $\sigma(E',E)\subseteq \TTT_{\BB_{E}}$. Since each $B=(\chi_n)\in \BB_{E'}$ is $\sigma(E',E)$-null, Lemma \ref{l:null-seq} implies that $B$ is $ \TTT_{\BB_{E}}$-precompact if and only if $B\cup\{0\}$ is $ \TTT_{\BB_{E}}$-compact.
Then the  equivalences (i)$\Leftrightarrow$(ii)$\Leftrightarrow$(iii)$\Leftrightarrow$(iv)$\Leftrightarrow$(v)$\Leftrightarrow$(vi) follow from %Theorem \ref{t:G-dual-top-1}
Theorem 9.2.1 of \cite{Edwards}  in which $\mathfrak{S}=\BB_E$ and $\mathfrak{S}'=\BB_{E'}$. %and the fact that all $A\in\BB_E$ and all $B\in \BB_{E'}$ are weakly compact (so that $\Id{\restriction}_A(A)$ is $\TTT_{\BB_{E'}}$-compact and $\Id^\ast{\restriction}_B(B)$ is $\TTT_{\BB_{E}}$-compact).
\smallskip

(vi)$\Ra$(vii) immediately follows from the fact that the weak topology for $\beta(E',E)$ on $E'=L'$ is stronger than the weak$^\ast$ topology $\sigma(E',E)$ on $E'$ and the continuity (with respect to $\sigma(E',E){\restriction}_B\times\sigma(E,E'){\restriction}_A$) of $\langle \eta,x\rangle$ at zero $(0,0)$.
\smallskip

(vii)$\Ra$(vi) Assume that $E$ has the sequential $DP_{(p,q)}$ property. We claim that the bilinear map $\langle \eta, x\rangle$ is weakly continuous at $(0,0)$ for  every $(x_n)\in \BB_E$ and each $(\chi_n)\in \BB_{E'}$. Indeed, suppose for a contradiction that there are $(x_n)\in \BB_E$ and  $(\chi_n)\in \BB_{E'}$ such that $\langle \eta, x\rangle$ is weakly discontinuous at $(0,0)$. Then there exists  $\e>0$ such that for every $i\in\w$ there are $n_i>i$ and $m_i>i$ such that $|\langle \eta_{n_i}, x_{m_i}\rangle|\geq \e$. Without loss of generality we can assume that $\{n_i\}_{i\in\w}$ and $\{m_i\}_{i\in\w}$ are strictly increasing. Then $(x_{n_i})\in \BB_E$ and $(\chi_{n_i})\in \BB_{E'}$ but $|\langle \eta_{n_i}, x_{m_i}\rangle|\not\to 0$, which contradicts the sequential $DP_{(p,q)}$ property of $E$.

Since all sequences $A=(x_n)\in \BB_E$ and  $B=(\chi_n)\in \BB_{E'}$ are weakly null, the claim implies that the bilinear map $\langle \eta, x\rangle$ is continuous on the {\em weakly compact} space $B\times A$. Therefore $\langle \eta, x\rangle$ is uniformly continuous on $B\times A$ for the product topology $\sigma(E',E){\restriction}_B\times\sigma(E,E'){\restriction}_A$.
\smallskip

(vii)$\Ra$(viii) Assume that $E$ has   the sequential $DP_{(p,q)}$ property. Suppose for a contradiction that there is a relatively weakly sequentially $p$-compact subset $A$ of $E$ which is not a  $q$-$(V^\ast)$ set. Then there exist a weakly $q$-summable sequence $\{\chi_n\}_{n\in\w}$ in $E'_\beta$ and $\e>0$ such that $\sup_{a\in A} |\langle\chi_n,a\rangle|\geq 2\e$ for every $n\in\w$. For each $n\in\w$, choose $a_n\in A$ such that $|\langle\chi_n,a_n\rangle|\geq \e$. Since $A$ is relatively weakly sequentially $p$-compact, the sequence $\{a_n\}_{n\in\w}$ has a subsequence $\{a_{n_k}\}_{k\in\w}$ which weakly $p$-converges to a point $z\in E$. Then
\[
\e\leq |\langle\chi_{n_k},a_{n_k}\rangle|\leq |\langle\chi_{n_k},a_{n_k}-z\rangle|+|\langle\chi_{n_k},z\rangle|\to 0 \;\; \mbox{ as $k\to\infty$},
\]
a contradiction.
\smallskip

(viii)$\Ra$(vii) Assume that each relatively weakly sequentially $p$-compact set in $E$ is a $q$-$(V^\ast)$ set. Let $\{x_n\}_{n\in\w}$ be a weakly $p$-summable sequence in $E$, and let $\{\chi_n\}_{n\in\w}$ be a weakly $q$-summable sequence in $E'_\beta$. Since the set $S:=\{x_n\}_{n\in\w}$ is weakly sequentially $p$-compact, it is a $q$-$(V^\ast)$ set. Therefore, by the definition of $q$-$(V^\ast)$ sets, we have
\[
0\leq \lim_{n\to\infty} |\langle\chi_n,x_n\rangle|\leq \lim_{n\to\infty} \sup_{i\in\w} |\langle\chi_n,x_i\rangle|=0.
\]
Thus $E$ has  the sequential $DP_{(p,q)}$ property.
\smallskip

Below we assume that $\TTT_{\BB_{E'}}$ is compatible with $\tau$. So $\sigma(E,E')$ is also the weak topology of $\big(E,\TTT_{\BB_{E'}}\big)$.
\smallskip

(iii)$\Ra$(ix) Let $(x_n)\in \ell_p^w\big(E,\TTT_{\BB_{E'}}\big)$ (or $\in c_0^w\big(E,\TTT_{\BB_{E'}}\big)$ if $p=\infty$). Since $\TTT_{\BB_{E'}}$ is compatible with $\tau$, it follows that $(x_n)\in \BB_{E}$. Then (iii) implies $x_n\to 0$ in $\TTT_{\BB_{E'}}$, that is, $\big(E,\TTT_{\BB_{E'}}\big)$ has the $p$-Schur property.

(ix)$\Ra$(iii) Let $A=(x_n)\in \BB_E$. Since $\TTT_{\BB_{E'}}$ and $\tau$ are compatible, it follows that $A\in \ell_p^w\big(E,\TTT_{\BB_{E'}}\big)$ (or $\in c_0^w\big(E,\TTT_{\BB_{E'}}\big)$ if $p=\infty$). Then the $p$-Schur property of $\big(E,\TTT_{\BB_{E'}}\big)$ implies that $\sigma(E,E'){\restriction}_A =\TTT_{\BB_{E'}}{\restriction}_A$. Taking into account that $A$ is (weakly) compact it follows that the identity map $\big(A,\sigma(E,E'){\restriction}_A\big)\to \big(E,\TTT_{\BB_{E'}}\big)$ is uniformly continuous.

Finally, assume that $E$ is such that $E'_\beta$ is locally complete (if $E$ is quasibarrelled, then by Proposition 11.2.3 of \cite{Jar}, $E'_\beta$ is quasi-complete and hence locally complete).
Then $\TTT_{\BB_{E'}}$ is compatible with $\tau$ by (iv) of Lemma \ref{l:Sigma-seq}.\qed
\end{proof}

%Proposition 14.12 %\ref{p:weak-sDPp}
%of \cite{Gab-Pel} states that each relatively weakly sequentially $p$-compact subset of an lcs $E$ is a $p$-$(V^\ast)$ set if and only if $\lim_{n\to\infty} \langle\chi_n,x_n\rangle=0$ for every weakly $p$-summable sequence $\{x_n\}_{n\in\w}$ in $ E$  and each weakly $p$-summable sequence $\{\chi_n\}_{n\in\w}$ in $ E'_\beta$. Therefore we obtain

Setting $p=q$ in Theorem \ref{t:seq-DPpq} we obtain
\begin{corollary} \label{c:weak-sDPp}
Let $p\in[1,\infty]$, and let $E$ be a locally convex space. Then $E$ has the sequential $DP_p$ property if and only if each relatively weakly sequentially $p$-compact subset of $E$  is a $p$-$(V^\ast)$ set.
\end{corollary}

Below we give several examples which show that all Dunford--Pettis type properties are distinct in general. In particular, we give the first example of a locally convex space $E$ with the strict $DP$ property and the sequential $DP$ property but which does not have the $DP$ property.
Recall that a Tychonoff space $X$ is called an {\em $F$-space} if every cozero-set $A$ in $X$ is $C^\ast$-embedded. For  numerous equivalent conditions for a Tychonoff space $X$ being an $F$-space see \cite[14.25]{GiJ}. In particular, the Stone--\v{C}ech compactification $\beta \Gamma$ of a discrete infinite space $\Gamma$ is a compact $F$-space.

\begin{example} \label{exa:SDP-non-DP}
Let $K$ be an infinite compact $F$-space, $E=C(K)$ and let $H:=\big( E', \mu(E',E)\big)$. Then:
\begin{enumerate}
\item[{\rm(i)}] $H$ is a complete, semi-reflexive, Schur space;
\item[{\rm(ii)}] $H$ is neither reflexive nor even quasibarrelled;
\item[{\rm(iii)}] $H$ has the strict $DP_p$ property for every  $p\in[1,\infty]$;
\item[{\rm(iv)}] $H$ has the sequential $DP_{(p,q)}$ property for each  $p,q\in[1,\infty]$;
\item[{\rm(v)}] $H$ does not have the $DP$ property.
\end{enumerate}
\end{example}

\begin{proof}
(i) The space $H$ is a complete  Schur space by Proposition 3.5 of \cite{Gabr-free-resp}. Since $H'_\beta=E_\beta$ and $E$ is a Banach space we obtain  $H'_\beta=E$ and hence $H''=H$ algebraically, that is $H$ is semi-reflexive.
\smallskip

(ii) As we showed in (i) $H'_\beta=E$. Since the Banach space $E=C(K)$ is not reflexive, the space $H$ also is not reflexive. The space $H$ is not quasibarrelled because, otherwise, it would be reflexive by Proposition 11.4.2 of \cite{Jar}.
\smallskip

(iii) By (v) of Lemma \ref{l:Sigma'}, we have $\tau_{\Sigma'}(H)=\mu(E',E)$. Therefore, by (i), the space $(H,\tau_{\Sigma'})$ has the Schur property and hence the $p$-Schur property. Thus, by Corollary \ref{c:semi-reflexive-strict-DP}, %Theorem \ref{t:strict-DPp},
$H$ has the strict $DP_p$ property.
\smallskip

(iv) %Since $E$ is a Banach space, any bounded subset of $H$ is contained in some $aB^\ast$, where $B^\ast$ is the polar of the closed unit ball $B$ of $E$. Therefore the strong dual of $H$ is the Banach space $E$.
Let $S$ be a weakly $q$-summable sequence in $E$. Then, by the Krein theorem, the bipolar $S^{\circ\circ}$ of $S$ is a weakly compact subset of $E=H'_\beta$. Therefore, by the Mackey--Arens theorem, the polar $S^\circ$ is a neighborhood of zero in $H$ and hence $S^{\circ\circ}$ is equicontinuous. Whence $S^{\circ\circ}\in \Sigma'(H')$ and hence the polar topology $\TTT_{\BB_{H'}}$ is weaker than $\tau_{\Sigma'}$ (recall that here $\BB_{H'}$  denotes the family of all weakly $q$-convergent sequences in $H'_\beta=E$). By (i) of Lemma \ref{l:Sigma-seq}, $\sigma(H,H')\subseteq \TTT_{\BB_{H'}}$. Since  $(H,\tau_{\Sigma'})$ has the $p$-Schur property (because $H$ has even the Schur property and $\tau_{\Sigma'}\subseteq \mu(E',E)$ by (i)  of Lemma \ref{l:Sigma'}), it follows that also the space $\big(H,\TTT_{\BB_{H'}}\big)$ has the  $p$-Schur property. Observe also that
\[
\sigma(H,H')\subseteq \TTT_{\BB_{H'}} \subseteq \tau_{\Sigma'}\subseteq \mu(E',E)
\]
and hence $\TTT_{\BB_{H'}}$ is compatible with the topology $\mu(E',E)$ of $H$. Thus, by Theorem \ref{t:seq-DPpq}, $H$ has the sequential $DP_{(p,q)}$ property.
\smallskip

(v) Let $B^\ast$ be the closed unit ball of the Banach dual $E'_\beta$. Then $B^\ast$ is a weakly compact subset of $H$ by the Alaoglu theorem. By Proposition 3.5 of \cite{Gabr-free-resp}, the space $H$ is a complete Schur space such that $B^\ast$ is not a compact subset of $H$ (so $H$ does not have the weak Glicksberg property). Taking into account that $H$ is complete, $B^\ast$ is not precompact in $H$. Since, by (iii) of Lemma \ref{l:Sigma'},  $\tau_{\Sigma'}(H)=\mu(E',E)$ we obtain that $B^\ast$ is not  $\tau_{\Sigma'}(H)$-precompact. Thus, by (iv) of  Theorem \ref{t:def-DP}, the space $H$ does not have the $DP$ property.\qed
%
%On the other hand, since $H$ is a Schur space every  weakly Cauchy sequence $S$ in $H$ is Cauchy for the original topology $\mu(E',E)$ of $H$. By (iii) of Lemma \ref{l:Sigma'}, we have $\tau_{\Sigma'}(H)=\mu(E',E)$. Therefore $S$ is Cauchy for the topology $\tau_{\Sigma'}(H)$. Thus, by (ii) of  Theorem \ref{t:DP}, the space $H$ has the $SDP$ property.\qed
\end{proof}

In Theorem 1.6 of \cite{Gabr-free-lcs} we proved that the free lcs $L(X)$ over any Tychonoff space $X$ has the $DP$ property. Below we generalize this result.

\begin{proposition} \label{p:L(X)-DP}
Let $p\in[1,\infty]$, and let $L(X)$ be the free locally convex space over a Tychonoff space $X$. Then:
\begin{enumerate}
\item[{\rm(i)}] $L(X)$  has the $DP$ property and the strict $DP_p$ property;
\item[{\rm(ii)}] every locally convex space $E$ is a quotient space of a space with the $DP$ property and the strict $DP_p$ property;
consequently, quotient spaces of spaces with the $DP$ property and the strict $DP_p$ property may not have the $DP$ property or the strict $DP_p$ property;
\item[{\rm(iii)}] if $X$ is a non-discrete metrizable space, then $L(X)$ does not have the sequential $DP$ property.
\end{enumerate}
\end{proposition}

\begin{proof}
(i) By Theorem 1.2 of \cite{Gab-Respected}, the space $L(X)$ has the Glicksberg property. Now Lemma \ref{l:weak-Glick} and Corollary \ref{c:weak-Glick-DP}  apply.

(ii) Consider the identity map $\Id_E:E\to E$. Then, by the definition of $L(E)$, the operator $\Id_E$ can be extended to an operator $T$ from $L(E)$ onto $E$. Since $E$ is a closed subspace of $L(E)$ it follows that $T$ is a quotient map. Finally, (i) applies.

(iii) This is Theorem 1.7 of \cite{Gabr-free-lcs}.\qed
\end{proof}

Although the weak Glicksberg property implies the $DP$ property (see (i) of Corollary \ref{c:weak-Glick-DP}), the converse is not true in general as the following example shows.

\begin{example} \label{exa:DP-non-weak-Glic}
The Banach space $C([0,\w])$ has the $DP$-property but it does not have the weak Glicksberg property.
\end{example}

\begin{proof}
By the Grothendieck theorem, $C([0,\w])$ has the $DP$-property.

To show that  $C([0,\w])$ does not have the weak Glicksberg property (we assume that the field is $\IR$), we consider the following subset of $C([0,\w])$
\[
K=\Big\{f\in C([0,\w]): \sum_{n\in\w} |f(n)|\leq 1\Big\}.
\]
It is clear that $K$ is absolutely convex. To show that $K$ is weakly compact, by the Eberlein--\v{S}mulian theorem, we shall prove that $K$ is weakly sequentially compact. Let $S=\{g_i\}_{i\in\w}$ be a sequence in $K$. Considering $S$ as a subset of the compact metric space $[-1,1]^{[0,\w]}$ without loss of generality we can assume that $S$ pointwise converges to a function $g\in [-1,1]^{[0,\w]}$. As $\sum_{n\in\w} |g_i(n)|\leq 1$ for every $i\in\w$, it follows that also  $\sum_{n\in\w} |g(n)|\leq 1$ and hence $g$ is continuous. The Lebesgue dominated convergence theorem implies  that $g_i $ weakly converges to $g$. Thus $K$ is an absolutely convex weakly compact subset of $C([0,\w])$.

On the other hand, for every $n\in\w$, let $f_n:[0,\w]\to \{0,1\}$ be the unique function such that $f^{-1}(1)=\{n\}$. Then $f_n\in K$. Since $\|f_n-f_m\|_{\infty}=1$ for all distinct $n,m\in\w$, it follows that $K$ is not a compact subset of $C([0,\w])$. Thus $C([0,\w])$ does not have the weak Glicksberg property.\qed
\end{proof}
In other words, Example \ref{exa:DP-non-weak-Glic} shows that there are even Banach spaces $E$  such that the space $(E,\tau_{\Sigma'})$ has the weak Glicksberg property, but the space $E$ does not have the weak Glicksberg property. On the other hand, Example \ref{exa:weak-Glick-non-Glick} below shows that Mackey spaces may have the weak Glicksberg property, but they do not have even the Schur property. To prove this example we need the next assertion.

\begin{proposition} \label{p:weak-Glick}
Let $E$ be a Mackey space which carries its weak topology. Then the space $H:=(E', \mu(E',E))$ has the weak Glicksberg property.
\end{proposition}

\begin{proof}
Let $K$ be an absolutely convex compact subset of $H_w=E'_{w^\ast}$. Since $E$ is a Mackey space, the Mackey--Arens theorem implies that the polar $K^\circ$ is a neighborhood of zero in $E$. Taking into account that the topology of $E$ is the weak topology we obtain that there is a finite subset $F$ of $E'$ such that $K$ is a compact subset of the finite-dimensional subspace $\spn(F)$ of $E'$. Therefore $K$ is also compact in  $H$. Thus $H$  has the weak Glicksberg property.\qed
\end{proof}

%The next example shows also that locally convex spaces with the weak Glicksberg property may not have even  the Schur property. %, but also that the Glicksberg property and the Schur property are not necessary to have the Dunford--Pettis property and the strict Dunford--Pettis property, respectively (see Propoposition \ref{p:weak-Glick-DP}).
\begin{example} \label{exa:weak-Glick-non-Glick}
Let $\mathbf{s}=[0,\w]$ be a convergent sequence, and let $L_\mu(\mathbf{s})$ be the free locally convex space $L(\mathbf{s})$ over $\mathbf{s}$ endowed with the Mackey topology. Then:
\begin{enumerate}
\item[{\rm(i)}] $L_\mu(\mathbf{s})$ has the weak Glicksberg property and hence the $DP$ property;
\item[{\rm(ii)}] $L_\mu(\mathbf{s})$ has neither the Schur property nor the strict $DP$ property.
\end{enumerate}
\end{example}

\begin{proof}
(i) Since the topology of $L(\mathbf{s})$ is compatible with the duality $\big(C_p(\mathbf{s}),L(\mathbf{s})\big)$ and the space $C_p(\mathbf{s})$ is a Mackey space with its weak topology, Proposition \ref{p:weak-Glick} implies that $L_\mu(\mathbf{s})$ has the weak Glicksberg property. Therefore, by (i) of Corollary \ref{c:weak-Glick-DP}, $L_\mu(\mathbf{s})$ has the  $DP$ property.

(ii) The space $L_\mu(\mathbf{s})$ does not have the Schur property by Remark 6.5 of \cite{Gab-Respected} (for a detailed proof, see Example 3.12 of \cite{BG-free}).
To show that $L_\mu(\mathbf{s})$ does not have the  strict $DP$ property we recall first that, by Proposition 3.4 of \cite{Gabr-free-lcs},  $\big(L_\mu(\mathbf{s})\big)'_\beta =C(\mathbf{s})$. Consider the set $K\subseteq C(\mathbf{s})$ defined in Example \ref{exa:DP-non-weak-Glic}, where it is proved that $K$ is an absolutely convex and weakly compact. Therefore $K$ is also weak$^\ast$ compact and hence, by the Mackey--Arens theorem, $K^\circ$ is a neighborhood of zero in $L_\mu(\mathbf{s})$. Whence $K=K^{\circ\circ}$ is equicontinuous. Hence $K\in \Sigma'\big(C(\mathbf{s})\big)$ and $K^\circ\in \tau_{\Sigma'}$.

Consider the sequence $\{\delta_n\}_{n\in\w}$ of Dirac measures in $L_\mu(\mathbf{s})$, where $\langle\delta_n,f\rangle=f(n)$ for $f\in C(\mathbf{s})$. If $f\in C(\mathbf{s})=L_\mu(\mathbf{s})'$, then $f(n)\to f(\omega)$ and hence $\langle f,\delta_n\rangle\to \langle f,\delta_\omega\rangle$. Therefore $\{\delta_n\}_{n\in\w}$ is weakly Cauchy in $L_\mu(\mathbf{s})$. On the other hand, for all distinct $n,m\in\w$,  if $f_n:[0,\w]\to \{0,1\}$ is the unique function such that $f^{-1}(1)=\{n\}$, then
\[
\sup_{f\in K} \big| \langle f,\delta_n-\delta_m\rangle\big| \geq  \big| \langle f_n,\delta_n-\delta_m\rangle\big| =1
\]
and hence $\delta_n-\delta_m\not\in \tfrac{1}{2} K^\circ\in \tau_{\Sigma'}$. Therefore the sequence $\{\delta_n\}_{n\in\w}$ is not Cauchy for the topology $\tau_{\Sigma'}$. By (ii) of Theorem \ref{t:DP}, the space $L_\mu(\mathbf{s})$  does not have the  strict $DP$ property. \qed
\end{proof}

It is natural to define the next version of the Dunford--Pettis property.

\begin{definition} \label{def:DP-weak}{\em
Let $p\in[1,\infty]$. A locally convex space $E$ is said to have the {\em weak Dunford--Pettis property of order $p$} (the {\em weak $DP_p$ property}) if for each Banach space $L$, any weakly compact operator $T\in\LL(E,L)$ is $p$-convergent.\qed}
\end{definition}

Below we characterize the weak $DP_p$ property.
\begin{theorem} \label{t:weak-DPp}
Let $p\in[1,\infty]$. For a locally convex space $E$ the following assertions are equivalent:
\begin{enumerate}
\item[{\rm(i)}] $E$ has the weak $DP_p$ property;
\item[{\rm(ii)}] for  every  Banach space $L$, each weakly compact operator $T\in\LL(E,L)$ transforms every weakly sequentially $p$-precompact set into a relatively compact subset of $L$;
\item[{\rm(iii)}] for  every  Banach space $L$, each weakly compact operator $T\in\LL(E,L)$ transforms every weakly sequentially $p$-compact set into a relatively compact subset of $L$.
\end{enumerate}
\end{theorem}

\begin{proof}
(i)$\Ra$(ii) Let $T:E\to L$  be a weakly compact  operator from $E$ into a Banach space $L$. Since $E$ has the weak $DP_p$ property, $T$ is $p$-convergent. Let $A$ be a weakly sequentially $p$-precompact subset of $E$. Then, by Proposition \ref{p:p-convergent-s}, $T(A)$ is  relatively sequentially compact in $L$. As $L$ is a Banach space, it follows that $T(A)$ is relatively compact in $L$, as desired.

(ii)$\Ra$(iii) is trivial.

(iii)$\Ra$(i) Let $T:E\to L$  be  a weakly compact  operator from $E$ into a Banach space $L$. Fix a weakly $p$-summable sequence $S=\{x_n\}_{n\in\w}$ in $E$. Since $S$ is weakly sequentially $p$-compact, by (iii), the set $T(S)=\{T(x_n)\}_{n\in\w}$ is relatively compact in $L$. As $S$ is a weakly null-sequence, Lemma \ref{l:null-seq} implies that $T(x_n)\to 0$ in the norm topology of $L$.  Thus $T$ is $p$-convergent.\qed
\end{proof}

%%%%%%%%%%%%%%%%%%%%%%%%%%%%%%%%%%%%%
%%%%%%%%%%%%%%%%%%%%%%%%%%%%%%%%%%%%%
%%%%%%%%%%%%%%%%%%%%%%%%%%%%%%%%%%%%%
%%%%%%%%%%%%%%%%%%%%%%%%%%%%%%%%%%%%%
%%%%%%%%%%%%%%%%%%%%%%%%%%%%%%%%%%%%%

\section{Permanent properties} \label{sec:perm-DP}

%%%%%%%%%%%%%%%%%%%%%%%%%%%%%%%%%%%%%
%%%%%%%%%%%%%%%%%%%%%%%%%%%%%%%%%%%%%
%%%%%%%%%%%%%%%%%%%%%%%%%%%%%%%%%%%%%
%%%%%%%%%%%%%%%%%%%%%%%%%%%%%%%%%%%%%
%%%%%%%%%%%%%%%%%%%%%%%%%%%%%%%%%%%%%

We start from the following lemma which summarazes some relationships between the Dunford--Pettis type properties introduced in Definition \ref{def:DP-general}.
\begin{lemma} \label{l:relations-DP-parameter}
Let $1\leq q\leq p\leq\infty$, $1\leq q'\leq p'\leq\infty$, and let $E$ be a locally convex space.
\begin{enumerate}
\item[{\rm(i)}]  $E$ has the strict $DP$ property if and only if it has the strict $DP_\infty$ property.
\item[{\rm(ii)}] If $E$ has the strict $DP_p$ property, then it has the weak $DP_p$ property. The converse is true if $E$ is a normed space or $E=E_w$.
\item[{\rm(iii)}] If $E$ is locally complete and has the $DP$ property, then it has the weak $DP_\infty$ property. %If $E$ is weakly separably von Neumann complete and has the $DP$ property, then it has the strict $DP$ property.
\item[{\rm(iv)}]  If $E$ has the strict {\rm(}resp., weak{\rm)} $DP_p$ property, then $E$ has the strict  {\rm(}resp., weak{\rm)} $DP_q$ property. %If $E$ is a weakly sequentially angelic space, then $E$ has the strict $DP$ property if and only if it has the $DP$ property.
\item[{\rm(v)}] If $E$ has the sequential $DP_{(p,p')}$ property, then $E$ has the sequential $DP_{(q,q')}$ property.
\end{enumerate}
\end{lemma}

\begin{proof}
(i) and (iv)-(v) follow from the corresponding definitions. %follows from Lemma \ref{l:p-conver-p-Cauchy}.

(ii) is clear. The converse assertion is true by Lemma \ref{l:weak-to-norm}.
%
%(iii) Let $T$ be an operator from $E$ into a Banach space $L$ which transforms bounded sets into relatively weakly compact sets. Assume that $E$ has the $DP$ property, and let $S=\{x_n\}_{n\in\w}$ be a weakly Cauchy sequence in $E$. Since $E$ is weakly separably von Neumann complete, the closed  absolutely convex hull $K$ of $S$ is weakly compact. As $E$ has the $DP$ property, $T(K)$ is a relatively compact subset of $L$. Therefore $T(S)$ is a relatively compact subset of $L$. Since $S$ is weakly Cauchy, $T(S)$ is a weakly Cauchy sequence in $L$. Hence $T(S)$ has a unique weak cluster point. Therefore $T(S)$ has a unique limit point and hence $T(S)$ is a convergent sequence. Thus $E$ has the strict $DP$ property.
%
%(iv) By (iii) we have to prove only the sufficiency. So, assume that $E$ has the strict $DP$ property. Let $K$ be an absolutely convex weakly compact subset of $E$. We have to show that $T(K)$ is a relatively compact subset of $L$. Since $L$ is a Banach space, it suffices to show that every sequence $\{T(x_n)\}_{n\in\w}$ in $T(K)$ has a convergent subsequence in $L$. As $E$ is weakly sequentially angelic, $K$ is a weakly sequentially compact subset of $E$, and hence the sequence $\{x_n\}_{n\in\w}$ has a subsequence $\{x_{n_k}\}_{k\in\w}$ weakly converging to some element $x\in K$. Since $E$ has the strict $DP$ property, we obtain that $T(x_{n_k})\to T(x)\in T(K)$. Thus $E$ has the $DP$ property.

(iii) Let $T$ be a weakly compact operator from $E$ to a Banach space $L$. Fix a weakly null sequence $S=\{x_n\}_{n\in\w}$ in $E$. Since $E$ is locally complete, $K:= \cacx(S)$ is a weakly compact subset of $E$. Then, by the $DP$ property of $E$, $T(K)$ is a relatively compact subset of $L$. Since $T(S)$ is a weakly null sequence in $L$, Lemma \ref{l:null-seq} implies that $T(x_n)\to 0$ in $L$. Thus $T$ is $\infty$-convergent. \qed
\end{proof}
%The following versus of the $(DP)$ property was also defined by Grothendieck \cite{Grothen}, and the weak$^\ast$ version of the $(RDP)$ property was defined in \cite{BLO}.

The next proposition shows some heredity of Dunford--Pettis type properties with respect to different (compatible) topologies.

\begin{proposition} \label{p:weak-top-DP}
Let $p,q\in[1,\infty]$, and let $(E,\tau)$ be a locally convex space.
\begin{enumerate}
\item[{\rm(i)}] If $\TTT$ is a locally convex vector topology on $E$  such that $\sigma(E,E')\subseteq \TTT\subseteq \tau$ and if $E$ has the $DP$ property $($resp., the strict $DP_p$ property or the weak $DP_p$ property$)$, then also $(E,\TTT)$ has the same property.
%\item[{\rm(ii)}] Let $\TTT$ be a locally convex vector topology on $E$  compatible with $\tau$, and assume that $E$ is barrelled. If $E$ has the $DP$ property $($resp., the strict $DP_p$ property  or the weak $DP_p$ property$)$,  then also $(E,\TTT)$ has the same property.
\item[{\rm(ii)}] If $E$ or $E'_\beta$ is feral, then $E$ has the sequential $DP_{(p,q)}$ property. In particular, $\IF^\kappa$ has the sequential $DP_{(p,q)}$ property for every cardinal number $\kappa$.
\item[{\rm(iii)}] If $\TTT$ is a locally convex vector topology on $E$ compatible with $\tau$, then $(E,\tau)$ has  the sequential $DP_{(p,q)}$ property if and only if $(E,\TTT)$ has the sequential $DP_{(p,q)}$ property.
\item[{\rm(iv)}]  If $E=E_w$, then $E$ has  the $DP$ property, the strict $DP_p$ property and the weak $DP_p$ property. If moreover $E$ is quasibarrelled, then $E$ has the sequential $DP_{(p,q)}$ property.
\end{enumerate}
\end{proposition}

\begin{proof}
(i) Let $T:(E,\TTT)\to L$ be an operator into a Banach space $L$ which transforms bounded sets into relatively weakly compact sets. As $\TTT\subseteq \tau$ it follows that $T:E\to L$ is continuous. The condition  $\sigma(E,E')\subseteq \TTT\subseteq \tau$ implies that $\TTT$ is compatible with $\tau$, and hence  $(E,\TTT)$ and $E$ have the same bounded sets. Therefore the operator $T:E\to L$ also  transforms bounded sets into relatively weakly compact sets.

Assume that  $E$ has the $DP$ property, and let $K$ be an arbitrary absolutely convex weakly compact sets of $(E,\TTT)$. Since $(E,\TTT)$ and $E$ have the same absolutely convex weakly compact sets, the $DP$ property of $E$ implies that $T(K)$ is a relatively compact subset of $L$. Thus $(E,\TTT)$ has the $DP$ property.

Assume that  $E$ has the strict $DP_p$ property. Then $T:E\to L$ is $p$-convergent. Since $(E,\TTT)$ and $E$ have the same weakly $p$-convergent sequences, it follows that $T:(E,\TTT)\to L$  is $p$-convergent. Thus $(E,\TTT)$ has the strict $DP_p$ property.

%Assume that $E$ has  the sequential $DP_p$ property. Then also $(E,\TTT)$ has  the sequential $DP_p$ property because $(E,\TTT)$ and $E$ have the same weakly $p$-summable sequences and $(E,\TTT)'_\beta=E'_\beta$.

Assume that  $E$ has the weak $DP_p$ property.  Let $T:(E,\TTT)\to L$ be a weakly compact operator into a Banach space $L$. Since $\TTT\subseteq \tau$, $T$ is continuous and weakly compact as an operator from $(E,\tau)$ into $L$. By  the weak $DP_p$ property of $E$, $T$ is $p$-convergent. The inclusions $\sigma(E,E')\subseteq \TTT\subseteq \tau$  imply that every weakly $p$-convergent sequence $\{x_n\}_{n\in\w}$ of $(E,\TTT)$ is also weakly $p$-convergent for $(E,\tau)$. Therefore $T(x_n)\to 0$. Thus $(E,\TTT)$ has the weak $DP_p$ property.
\smallskip

%(ii) Let $T:(E,\TTT)\to L$ be an operator to a Banach space $L$  which transforms bounded sets into relatively weakly compact sets (resp., weakly compact). Then, by Theorem 8.11.3 of \cite{NaB}, $T$ is weakly continuous. Since $E$ is barrelled and $\TTT$ is  compatible with $\tau$, Corollary 11.3.7 of \cite{NaB} guarantees that $T$ is continuous   as a linear map from $(E,\tau)$ to $L$. Therefore, by the $DP$ property (resp., the strict $DP_p$ property), Now the proof of the clause (i) shows that $(E,\TTT)$ has the $DP$ property (resp., the strict $DP_p$ property  or the weak $DP_p$ property), $T(K)$ is a relatively compact subset of $L$ for every absolutely convex weakly compact sets of $E$ (resp., the operator $T:(E,\tau)\to L$ is $p$-convergent). It remains to note that $(E,\TTT)$ and $E$ have the same weakly $p$-convergent sequences and weakly compact subsets. However, it does not follow that $T:(E,\TTT)\to L$ is $p$-convergent!!!!!
%\smallskip

(ii) Assume that $E$ or $E'_\beta$ is feral. We consider only the case when $E'_\beta$ is feral (the case when $E$ is feral is considered analogously). Let $S=\{ x_n\}_{n\in\w}$ be a  weakly $p$-summable sequence in $E$, and let $\{ \chi_n\}_{n\in\w}$ be a weakly $q$-summable sequence  in  $E'_\beta$. Since  $E'_\beta$ is feral, $\{ \chi_n\}_{n\in\w}$ is finite-dimensional and hence there are $\eta_1,\dots,\eta_k\in E'$ and $a_{n,1},\dots,a_{n,k}\in\IF$ such that
\[
\chi_n= a_{n,1}\eta_1 +\cdots + a_{n,k}\eta_k \quad \mbox{ for every $n\in\w$},
\]
and $\lim_{n\to\infty}a_{n,i}= 0$ for every $1\leq i\leq k$. Since $S$ is bounded, there is $C>0$ such that $\sup_{n\in\w} |\langle\eta_i, x_n\rangle| <C$ for every $1\leq i\leq k$. Then
\[
|\langle\chi_n,x_n\rangle| \leq\sum_{i=1}^k |a_{n,i}|\cdot |\langle\eta_i, x_n\rangle| \leq C \cdot \sum_{i=1}^k |a_{n,i}| \to 0,
\]
which means that $E$ has the sequential $DP_{(p,q)}$ property.

(iii) follows from the equality $(E,\TTT)'_\beta=E'_\beta$ and the fact that $(E,\TTT)$ and $E$ have the same weakly $p$-summable sequences.

(iv) By Lemma \ref{l:weak-to-norm}, {\em every} operator $T:E\to L$  into a Banach space $L$ is finite-dimensional. Hence $T$ trivially transforms absolutely convex weakly compact sets of $E$ into relatively compact sets of $L$.  Thus $E$ has  the $DP$ property.
%By Proposition \ref{p:weak*-p-convergent-weak-con},
Since $T$ is finite-dimensional, $T$ is also $p$-convergent. Thus $E$ has the strict $DP_p$ property and hence also the weak $DP_p$ property.

Assume that $E$ is quasibarrelled. Then, by Proposition \ref{p:strong-dual-feral}, the space $E'_\beta$ is feral. Thus, by  (ii), $E$ has the sequential $DP_{(p,q)}$ property. \qed
\end{proof}

 The next corollary essentially generalizes (i), (ii) and (v) of Theorem \ref{t:function-DP}.
\begin{corollary} \label{c:DP-weaker}
Let $p,q\in[1,\infty]$, $X$ be a Tychonoff space, and let $E$ be a subspace of the product $\IF^X$ containing $C_p(X)$. Then $E$ has the $DP$ property, the strict $DP_p$ property and the sequential $DP_{(p,q)}$ property.
\end{corollary}

\begin{proof}
Since $C_p(X)$ is quasibarrelled and carries its weak topology, Proposition \ref{p:strong-dual-feral} implies that also $E$ is quasibarrelled. Now (iv) of Proposition \ref{p:weak-top-DP} applies. \qed
\end{proof}

In Example 9.4.2 of \cite{Edwards} it is shown that a semi-reflexive space has the $DP$ property if and only if it is semi-Montel (= boundedly compact), and it is pointed out without a proof that any semi-Montel space has the strict $DP$ property. Below we give  a different proof of the first result, generalize the second one and obtain a sufficient condition when all these assertions are equivalent.

\begin{proposition} \label{p:semi-ref-DP}
Let $p\in[1,\infty]$. For a semi-reflexive locally convex space  $(E,\tau)$ consider the following assertions:
\begin{enumerate}
\item[{\rm(i)}] $E$ has the $DP$ property;
\item[{\rm(ii)}]  $E$ is semi-Montel;
\item[{\rm(iii)}] $E$ has the strict $DP_p$ property;
\item[{\rm(iv)}] $E$ has the $p$-Schur property.
\end{enumerate}
Then {\rm(i)$\Leftrightarrow$(ii)$\Ra$(iii)$\Leftrightarrow$(iv)}. If in addition $E$ is a weakly sequentially $p$-angelic space, then  all assertions {\rm(i)-(iv)} are equivalent.
\end{proposition}

\begin{proof}
Recall that $\tau_{\Sigma'}= \tau$, see (iv) of Lemma \ref{l:Sigma'}. Observe also that if $B$ is a bounded subset of $E$, then $B^\circ$ is a neighborhood of zero in $E'_\beta$. As $E$ is semi-reflexive, the Alaoglu theorem implies that the bipolar $K:=B^{\circ\circ}$ is an absolutely convex, weakly compact subset of $E$. Note also that since $E$ is locally convex it has a base consisting of weakly closed sets (namely $\Nn_0^c(E)$). As $K$ is weakly complete, Theorem 3.2.4 of \cite{Jar} implies that $K$ is $\tau$-complete.

(i)$\Ra$(ii) Assume that $E$ has the $DP$ property, and let $B$ be a bounded subset of $E$. By (iii) of Theorem \ref{t:def-DP}, the bipolar $K:=B^{\circ\circ}$ is precompact for the topology $\tau_{\Sigma'}= \tau$. As we noticed above $K$ is also $\tau$-complete and hence $K$ is a compact subset of $E$. Thus $B\subseteq K$ is relatively compact and hence $E$ is semi-Montel.

(ii)$\Ra$(i) and (ii)$\Ra$(iii)  follow from (iii) of Corollary \ref{c:weak-Glick-DP}. %\ref{c:semi-reflexive-strict-DP}.%If $E$ is semi-Montel, then every absolutely convex, weakly compact subset of $E$ is compact and hence $\tau_{\Sigma'}$-precompact. Thus, by (i) of Theorem \ref{t:DP}, $E$ has the $DP$ property.

%(ii)$\Ra$(iii) Let $\{x_n\}_{n\in\w}\subseteq E$ be a weakly $p$-summable sequence. Set $S:=\{x_n\}_{n\in\w}\cup\{0\}$. Then $S$ is a weakly compact subset of $E$. Since $E$ is semi-Montel, $S$ is a compact subset of $E$  and hence $x_n\to 0$ in $E$. Therefore for any operator $T$ from $E$ into a Banach space $L$ we have $T(x_n)\to 0$. Thus $T$ is $p$-convergent and hence $E$ has the strict $DP_p$ property.

(iii)$\Leftrightarrow$(iv) follows from the equality $\tau_{\Sigma'}= \tau$ and Theorem \ref{t:strict-DPp}.
\smallskip

Assume that $E$ is a weakly sequentially $p$-angelic space. To show that all assertions (i)-(iv) are equivalent it suffices to prove the implication (iii)$\Ra$(ii). So we assume that $E$ has the strict $DP_p$ property. Let $B$ be a bounded subset of $E$. Set $K:=B^{\circ\circ}$. As we noticed above $K$ is an absolutely convex, weakly compact subset of $E$ which is $\tau$-complete. To show that $E$ is semi-Montel we prove that $K$ is also a precompact subset of $E$ (and then $K$ is compact because it is $\tau$-complete).

Suppose for a contradiction that $K$ is not precompact. Then there is a sequence $S=\{x_n\}_{n\in\w}$ in $K$ and $U\in\Nn^c_0(E)$ such that $x_n-x_m\not\in U$ for all distinct $n,m\in\w$. Since $E$ is weakly sequentially $p$-angelic, $K$ is a relatively weakly sequentially $p$-compact subset of $E$. Therefore we can assume also that $x_n$ weakly $p$-converges to a point $x\in E$. As $K$ is weakly compact we also have $x\in K$.

Set $N(U):=\bigcap_{n\in\NN} \tfrac{1}{n} U$. It easy to see that $N(U)+U=U$. It is well known that the gauge functional $q_U$ of $U$ defines a norm on the quotient space $E_{(U)}:=E/N(U)$. Let $\Phi_U: E\to E_{(U)}$ be the quotient map. Then $\Phi_U$  is continuous and $\Phi_U(U)$ is the closed unit ball $B_{E_{(U)}}$ of the normed space $E_{(U)}$. Let $L$ be a completion of $E_{(U)}$. We shall consider $\Phi_U$ also as an operator from $E$ into $L$. Now, let $D$ be a bounded subset of $E$. Then, as we noticed above,  $D^{\circ\circ}$ is  an absolutely convex, weakly compact subset of $E$. Since $\Phi_U$ is also weakly continuous, it follows that $\Phi_U\big(D^{\circ\circ}\big)$ is a weakly compact subset of $L$. Therefore the operator $\Phi_U$ satisfies the definition of the strict $DP_p$ property and hence $\Phi_U$ is a $p$-convergent operator. Whence $\Phi_U(x_n)\to \Phi_U(x)$ in $L$. However, since $N(U)+U= U$ and $x_n-x_m\not\in U$, we obtain $\Phi_U(x_n) - \Phi_U(x_m)\not\in \Phi_U(U)=B_{E_{(U)}}$ and hence $\Phi_U(x_n)\not\to \Phi_U(x)$, a contradiction.\qed
\end{proof}

It is known that any Fr\'{e}chet--Montel space is separable, see Theorem 11.6.2 of \cite{Jar}. This result and Proposition \ref{p:semi-ref-DP} imply

\begin{corollary} \label{c:refl-Frechet-DP}
Let $E$ be a reflexive Fr\'{e}chet space. Then $E$ has the $DP$ property if and only if it is a Montel space. In particular, $E$ is separable. Consequently, a reflexive infinite-dimension Banach space does not have the $DP$ property.
\end{corollary}

\begin{remark} {\em
Grothendieck proved in Proposition 1 and Corollaire of \cite{Grothen} (see also \cite[9.4.3(e)]{Edwards}) the next assertion:  if $E$ is quasibarrelled whose strong dual posseses the $DP$ property, then so too does $E$. We note that the condition of being quasibarrelled is not necessary in this assertion. Indeed, let $K$ be a metric compact space, then the Banach space $C(K)$ has the $DP$ property. By  Proposition \ref{p:L(X)-DP}, the free lcs $L(K)$ over $K$ also has the $DP$ property. By Proposition 3.4 of  \cite{Gabr-free-lcs}, $C(K)$ is the strong dual of $L(K)$. However, by \cite{Gabr-L(X)-Mackey}, $L(K)$ is not even a Mackey space. Observe also that, by Theorem 1.7 of  \cite{Gabr-free-lcs}, the space $L(K)$ does not have the sequential $DP$ property. \qed}
\end{remark}

It is known that the product of a family of locally convex spaces has the $DP$ property (or the strict $DP$ property) if and only if every factor  has the same property, see \cite[9.4.3]{Edwards}. Below we extend this result.

\begin{proposition} \label{p:DP-product}
Let $p,q\in[1,\infty]$, and let $\{E_i\}_{i\in I}$ be a nonempty family of locally convex spaces.
\begin{enumerate}
\item[{\rm(i)}] $E=\prod_{i\in I} E_i$ has the $DP$ property $($resp., the strict $DP_p$ property, the weak $DP_p$ property or the sequential $DP_{(p,q)}$ property,$)$ if and only if for every $i\in I$, the factor $E_i$ has the same property.
\item[{\rm(ii)}] $E=\bigoplus_{i\in I} E_i$ has the $DP$ property $($resp., the strict $DP_p$ property, the weak $DP_p$ property or the sequential $DP_{(p,q)}$ property$)$ % the $[$strong$]$ $DP_{(p,q)}^\ast$ property or  the $[$strong$]$ $DPV_{(p,q)}^\ast$ property$)$
    if and only if for every $i\in I$, the factor $E_i$ has the same property.
%\item[{\rm(iii)}] If $I=\w$ is countable, then $E=\prod_{i\in \w} E_i$ has  the $[$strong$]$ $DP_{(p,q)}^\ast$ property $($resp., the $[$strong$]$ $DPV_{(p,q)}^\ast$ property$)$  if and only if all factors $E_i$ have the same property.
\end{enumerate}
\end{proposition}

\begin{proof}
Since the case of  the $DP$ property was considered in \cite[9.4.3]{Edwards}, below we consider other types of the Dunford--Pettis property.

{\em I. The strict $DP_p$ property and the weak $DP_p$ property}.
Assume that $E$ has  the strict $DP_p$ property (resp., the weak $DP_p$ property). Fix an arbitrary $j\in I$, and let $T$ be an operator from $E_j$ to a Banach space $L$ which transforms bounded sets into relatively weakly compact sets (resp., $T$ is weakly compact). Denote by $\bar T$ the operator from $E$ into $L$ defined by
\[
\bar T (x):= T(x_j), \;\; \mbox{ where $x=(x_i)\in E$}.
\]
It is clear that $\bar T$ transforms bounded sets into relatively weakly compact sets (resp., $\overline{T}$ is weakly compact). Since $E$ has  the strict $DP_p$ property (resp.,  the weak $DP_p$ property), the operator $\bar T$ is $p$-convergent. Therefore $T$ is also $p$-convergent. Thus $E_j$ has  the strict $DP_p$ property (resp.,  the weak $DP_p$ property).

Conversely, assume that all spaces $E_i$ have  the strict $DP_p$ property (resp.,  the weak $DP_p$ property). Let $T$ be an operator from $E$ to a Banach space $L$ which transforms bounded sets into relatively weakly compact sets (resp., $T$ is weakly compact). We have to show that $T$ is $p$-convergent. To this end, let $\big\{(x^n_i)_{i\in I}\big\}_{n\in\w}$ be a weakly $p$-summable sequence in $E$. Now we distinguish between the cases (i) and (ii).

{\em Case (i)}.  Since any neighborhood of zero in $E$ contains a subspace of the form $\prod_{i\in F} \{0\}\times \prod_{i\in I\SM F} E_i$ for some finite $F\subseteq I$ and since the unit ball of $L$ has no non-trivial linear subspaces, the continuity of $T$ implies that there is a finite subset $F$ of $I$ such that $T\big((x_i)_{i\in I}\big)=\sum_{j\in F} T(\xxx_j)$, where $\xxx_j:=(y_i)$ is such that $y_j=x_j$ and $y_i=0$ for $i\not=j$. By (ii) of Lemma \ref{l:support-p-sum}, for every $j\in F$, the sequence $\{x^n_j\}_{n\in\w}$ is weakly $p$-summable in $E_j$.
It is clear that $T(\xxx_j)$ transforms bounded sets of $E_j$ into relatively weakly compact sets of $L$. Since $E_j$ has the strict $DP_p$ property (resp.,  the weak $DP_p$ property) it follows that $\lim_{n\to\infty} T(\xxx_j^n)=0$. As $F$ is finite we obtain that $T\big((x_i^n)\big)\to 0$ in $L$. Thus $T$ is $p$-convergent.

{\em Case (ii).} By (iv) of Lemma \ref{l:support-p-sum}, the support $F$ of $\big\{(x^n_i)_{i\in I}\big\}_{n\in\w}$ is finite and, for every $j\in F$, the sequence $\{x^n_j\}_{n\in\w}$ is weakly $p$-summable in $E_j$.
It is clear that the restriction $T{\restriction}_{E_j}$ transforms bounded sets into relatively weakly compact sets (resp., $T{\restriction}_{E_j}$ is weakly compact). Hence, by the strict $DP_p$ property (resp.,  the weak $DP_p$ property), $T{\restriction}_{E_j}$ is $p$-convergent. Therefore
\[
T\big((x_i)_{i\in I}\big)=\sum_{j\in F} T{\restriction}_{E_j}(x^n_j) \to 0.
\]
Thus $T$ is $p$-convergent.
\smallskip

{\em II. The sequential $DP_{(p,q)}$ property.}  The cases (i) and (ii) for the sequential $DP_{(p,q)}$ property immediately follow from (i)-(ii) and (iii)-(iv) of Lemma \ref{l:support-p-sum}, respectively, and the definition of the sequential $DP_{(p,q)}$ property.\qed
%\smallskip
%
%
%{\em III. The (strong) $DP_{(p,q)}^\ast$ property and the (strong) $DPV_{(p,q)}^\ast$ property.}  The cases (ii) and (iii) for these properties  follow from  Lemma \ref{l:seq-p-comp-prod} and Propositions \ref{p:product-sum-limited-set} and \ref{p:product-sum-V*-set}, respectively.\qed
\end{proof}

%\begin{remark} \label{rem:DP-hereditary} {\em
%Locally convex space with the $DP$ property (resp., the strict $DP_p$ property  or the sequential $DP_{(p,q)}$ property) may contain dense subspaces without that property. Indeed, let $X$ be a Banach space without one of the mentioned properties. Then, by (i) and (iv) of Proposition \ref{p:weak-top-DP}, the space $E=X_w$ does not have the $DP$ property (resp., the strict $DP_p$ property  or the sequential $DP_{(p,q)}$ property). However, the completion of $E$ being isomorphic to some product $\IF^\kappa$ has all those properties by (i) of Proposition \ref{p:DP-product} and (v) of Proposition \ref{p:weak-top-DP}.\qed}
%\end{remark}

% {l:Sigma'}If $H$ is a large subspace of $E$, then $(H,\tau_{\Sigma'}(H))$ is a subspace of $(E,\tau_{\Sigma'}(E))$.
The next assertion shows that the Dunford--Pettis type properties are inherited by large subspaces.
\begin{proposition} \label{p:DP-hereditary-large}
Let $p\in[1,\infty]$, and let $H$ be a large subspace of a locally convex space $E$.
%\begin{enumerate}
If $E$ has the $DP$ property {\rm(}resp., the strict $DP_p$ property{\rm)}, then also $H$ has the same property.
%\item[{\rm(ii)}] Assume that $E$ is a $q$-quasibarrelled space whose strong dual $E'_\beta$ is  locally complete {\rm(}for example, $E$ is quasibarrelled{\rm)}. If $E$ has the sequential $DP_{(p,q)}$ property, then also $H$ has this property.
%\end{enumerate}
\end{proposition}

\begin{proof}
Recall that, by (vii) of Lemma \ref{l:Sigma'}, $(H,\tau_{\Sigma'}(H))$ is a subspace of $(E,\tau_{\Sigma'}(E))$.

Assume that $E$ has the $DP$ property.
Let $K$ be an absolutely convex weakly compact subset of $(H,\tau_{\Sigma'}(H))$. Then $K$ is  an absolutely convex weakly compact subset of  $(E,\tau_{\Sigma'}(E))$. Since $E$ has the $DP$ property, (viii) of Theorem \ref{t:def-DP} implies that $K$ is a compact subset of $(E,\tau_{\Sigma'}(E))$ and hence of $(H,\tau_{\Sigma'}(H))$. Therefore $(H,\tau_{\Sigma'}(H))$ has the weak Glicksberg property. Once more applying  Theorem \ref{t:def-DP} we obtain that $H$ has the  $DP$ property.

Assume that $E$ has the strict $DP_p$ property. Then, by (vi) of Theorem \ref{t:strict-DPp}, $(E,\tau_{\Sigma'}(E))$ has the $p$-Schur property. Therefore also $(H,\tau_{\Sigma'}(H))$ has  the $p$-Schur property. Thus, by (vi) of Theorem \ref{t:strict-DPp}, $H$ has the strict $DP_p$ property.\qed
%
%(ii) Since $E$ is a $q$-quasibarrelled space whose strong dual $E'_\beta$ is  locally complete, Lemma \ref{l:Sigma-seq} implies $\tau_{S_q}\subseteq \tau_{\Sigma'}$. Therefore, by (viii) of Theorem \ref{t:seq-DPpq}, the space $(E,\tau_{S_q})$ has the $p$-Schur property.
\end{proof}

%{l:Sigma-seq}$\tau_{S_q}\subseteq \tau_{\Sigma'}$ if and only if $E$ is a $q$-quasibarrelled space whose strong dual $E'_\beta$ is  locally complete {\rm(}for example, $E$ is quasibarrelled{\rm)}

\begin{corollary} \label{c:DP-hereditary-met}
Let $p\in[1,\infty]$, and let $H$ be either a normed space or a separable metrizable locally convex space. If $E$ has the $DP$ property {\rm(}resp., the strict $DP_p$ property{\rm)}, then any dense subspace $H$ of $E$ has the same property.
\end{corollary}

\begin{proof}
If $E$ is a noremd space, then clearly $H$ is large. If $E$ is a separable metrizable locally convex space, then $H$ is also large by \cite[\S~29.6(1)]{Kothe}. Thus $H$ has the $DP$ property (resp., the strict $DP_p$ property) by Proposition \ref{p:DP-hereditary-large}.\qed
\end{proof}

%\begin{lemma} \label{l:p-q-Schur}
%Let $1\leq q\leq p\leq\infty$, $1\leq q'\leq p'\leq\infty$,  and let $E$ be a locally convex space.
%\begin{enumerate}
%\item[{\rm(i)}] If $E$ has the $p$-Schur property, then $E$ has the $q$-Schur property.
%\item[{\rm(ii)}] If $E$ has the strict $DP_p$ property, then $E$ has the strict $DP_q$ property.
%\item[{\rm(iii)}] If $E$ has the sequential $DP_{(p,p')}$ property, then $E$ has the sequential $DP_{(q,q')}$ property.
%\end{enumerate}
%\end{lemma}

%{l:null-seq}

It is known (see \cite[9.4.3(d)]{Edwards}) that if $E$ is a metrizable locally convex space, then the strict $DP$ property implies the $DP$ property. If in addition $E$ is a weakly sequentially complete space then the strict $DP$ property and the $DP$ property are equivalent. The next assertion shows that the weak sequential completeness is essential even for normed non-complete spaces. Below as usual $\tfrac{1}{p}+\tfrac{1}{p^\ast}=1$.

%{p:Lp-Schur}
\begin{proposition} \label{p:Lp-DP}
Let $1< p<\infty$, and let $q,q'\in[1,\infty]$. Then:
\begin{enumerate}
\item[{\rm(i)}] if $q<p^\ast$, then $\ell_p$ and $\ell_p^0$ have  the strict $DP_q$ property  and the sequential $DP_{(q,q')}$ property;
\item[{\rm(ii)}] if $q\geq p^\ast$, then the spaces $\ell_p$  and $\ell_p^0$ have neither the strict $DP_q$ property nor  the sequential $DP_{(q,q')}$ property;
\item[{\rm(iii)}] $\ell_p$ does not have the $DP$ property;
\item[{\rm(iv)}] $\ell_p^0$ has the weak Glicksberg property and hence it has the $DP$ property; consequently, a completion of even a normed space with the $DP$ property may not have the $DP$ property.
%\item[{\rm(v)}] $\ell_p^0$ has the $DP$ property.
%\item[{\rm(vi)}] $\ell_p$ is weakly sequentially $q$-angelic if and only if $q\leq p^\ast$.
\end{enumerate}
\end{proposition}

\begin{proof}
(i) By (i) of Lemma \ref{l:Sigma'}, we have  $\sigma(E,E')\subseteq \tau_{\Sigma'}\subseteq\tau$, where $\tau$ is the norm topology of $\ell_p$. By Proposition \ref{p:Lp-Schur}, the space $\ell_p$ and hence also $\ell_p^0$ have the $q$-Schur property. By (v) of Lemma \ref{l:Sigma-seq}, we have $\tau_{S_q}\subseteq \tau_{\Sigma'}$. Therefore the spaces $(E,\tau_{\Sigma'})$ and $(E,\tau_{S_q})$, where $E=\ell_p$ or $E=\ell_p^0$,  have the $q$-Schur property.   Taking into account that $\ell_p$ and $\ell_p^0$ are quasibarrelled, the spaces  $\ell_p$ and $\ell_p^0$ have the strict $DP_q$ property  and the sequential $DP_{(q,q')}$ property by (vi) of Theorem \ref{t:strict-DPp}  and (viii) of Theorem \ref{t:seq-DPpq}, respectively.

%As $\ell_p$ is a Banach space and, by (i), has the $q$-Schur property, Theorem \ref{t:seq-DPpq} implies that $\ell_p$ has the sequential $DP_{(q,q')}$ property.

%For every $U\in\Nn_0(\ell_p^0)$, the polar $U^\circ$ is weakly compact in $\ell_{p^\ast}$ because the space  $\ell_{p^\ast}$ is reflexive. Therefore, by (iii) of Lemma \ref{l:Sigma'}, we have $\tau_{\Sigma'}=\tau$, where $\tau$ is the norm topology on $\ell_p^0$.  Thus the space $\ell_p^0$ has the strict $DP_q$ property by Theorem \ref{t:strict-DPp}.

%By (i), the space $\ell_p^0$ has the $q$-Schur property. Taking into account that $\big(\ell_p^0\big)'_\beta =\ell_{p^\ast}$ is a Banach space, the equivalence (vii)$\Leftrightarrow$(viii) of Theorem \ref{t:seq-DPpq} implies that $\ell_p^0$ has the sequential $DP_{(q,q')}$ property.
\smallskip

(ii) Since $\ell_p^0$ is a subspace of $\ell_p$, Proposition \ref{p:Lp-Schur} implies that also the space $\ell_p$ does not have the $q$-Schur property. The reflexivity of $\ell_p$ and (iv) of Lemma \ref{l:Sigma'} imply $\tau_{\Sigma'}=\tau$, where $\tau$ is the norm topology on $\ell_p$. Thus, by  Theorem \ref{t:strict-DPp}, $\ell_p$ does not have the strict $DP_q$ property.
As $\ell_p$ is a Banach space and does not have the $q$-Schur property, (ix) of Theorem \ref{t:seq-DPpq} implies that $\ell_p$ does not have the sequential $DP_{(q,q')}$ property.

Analogously, by (vi) of Lemma \ref{l:Sigma'}, for the normed space $\ell_p^0$   we have $\tau_{\Sigma'}=\tau$, where $\tau$ is the norm topology on $\ell_p^0$. Proposition \ref{p:Lp-Schur} implies that also the space $\ell_p^0$ does not have the $q$-Schur property. Therefore,  Theorem \ref{t:strict-DPp} implies that $\ell_p^0$ does not have the strict $DP_q$ property.
Taking into account that $\big(\ell_p^0\big)'_\beta =\ell_{p^\ast}$ is a Banach space, %Proposition \ref{p:Lp-Schur} and
the equivalence (vii)$\Leftrightarrow$(ix) of Theorem \ref{t:seq-DPpq} implies that $\ell_p^0$ does not have the sequential $DP_{(q,q')}$ property.
\smallskip

(iii) follows from Corollary \ref{c:refl-Frechet-DP}. Another proof: By Proposition \ref{p:Lp-q-angelic}, the space $\ell_p$ is weakly sequentially $p^\ast$-angelic. By (ii), $\ell_p$ does not have the strict $DP_{p^\ast}$ property. Therefore, by Proposition \ref{p:semi-ref-DP}, $\ell_p$ does not have the $DP$ property.
\smallskip

(iv) Let $K$ be an absolutely convex, weakly compact subset of $\ell_p^0$, and suppose for a contradiction that $K$ is not a compact subset of $\ell_p^0$.

{\em Claim 1. There are $\e>0$ and a sequence $S=\{x_n\}_{n\in\w}$ in $K$ such that
\begin{equation} \label{equ:compact-ell-p-0}
\|x_n-x_m\|_{\ell_p}\geq \e\;\; \mbox{ and } \;\;\|x_n\|_{\ell_p}\geq 2\e
\end{equation}
for all distinct $n,m\in\w$.}

Indeed, since $K$ is also a weakly compact subset of $\ell_p$, we obtain that $K$ is a closed and non-compact subset of $\ell_p$. Therefore the description of compact subsets of $\ell_p$ given in  \cite[p.~6]{Diestel}
%Proposition \ref{p:compact-ell-p}
 implies that there are $\e>0$ and a strictly increasing subsequence $(m_k)\subseteq \w$ such that
\begin{equation} \label{equ:compact-ell-p}
\sup\Big\{ \sum_{m_k\leq n} |x_n|^p : x=(x_n)\in K\Big\} >(2\e)^p \;\; \mbox{ for every $k\in\w$}.
\end{equation}
Now we proceed by induction. For $n=0$, choose $x_0=(x_{n,0})\in K$ such that
\[
\|x_0\|_{\ell_p}^p\geq \sum_{m_0\leq n} |x_{n,0}|^p> (2\e)^p \quad (\mbox{so that } \|x_0\|_{\ell_p}>2\e).
\]
Assume that we found $x_0,\dots, x_{n-1}\in K$ which satisfy (\ref{equ:compact-ell-p-0}). Choose $m_n$ such that $\sum_{m_n\leq i} |x_{i,j}|^p< \e^p$ for every $0\leq j\leq n-1$. Then, by (\ref{equ:compact-ell-p}), there is $x_n=(x_{i,n})\in K$ such that
\[
\|x_n\|_{\ell_p}^p\geq \sum_{m_n\leq i} |x_{i,n}|^p> (2\e)^p \quad (\mbox{so that } \|x_n\|_{\ell_p}>2\e).
\]
Then, for every $0\leq j\leq n-1$, we obtain
\[
\|x_n-x_j\|_{\ell_p}\geq \Big(\sum_{m_n\leq i} |x_{i,n}-x_{i,j}|^p\Big)^{1/p}\geq \Big(\sum_{m_n\leq i} |x_{i,n}|^p\Big)^{1/p} -\Big(\sum_{m_n\leq i} |x_{i,j}|^p\Big)^{1/p}>\e.
\]
This completes the inductive process. The claim is proved.

%Then, by (i) of Proposition \ref{p:compact-ell-p}, there are $\e>0$ and a sequence $S=\{x_n\}_{n\in\w}$ in $K$ such that $\|x_n-x_m\|_{\ell_p}\geq \e$ and $\|x_n\|_{\ell_p}\geq \e$ for all distinct $n,m\in\w$.

Since $K$ in the weak topology is metrizable, without loss of generality we assume that $S$ weakly converges to some element $x\in K$. Denote by $(e_k)$ the standard unit basis of $\ell_{p}^0$. Hence the sequence $\{x_n-x\}_{n\in\w}$ is weakly null and (\ref{equ:compact-ell-p-0}) implies
\[
\|(x_n-x)-(x_m-x)\|_{\ell_p}\geq\e
\]
for all distinct $n,m\in\w$. For every $n\in\w$, set $b_n:=\|x_{n}-x\|_{\ell_{p}}$. Observe that if $\|x_{n_k}-x\|_{\ell_{p}}\to 0$ for some sequence $(n_k)\subseteq \w$, then
\[
\|x_{n_k}-x_{n_j}\|_{\ell_{p}}\leq \|x_{n_k}-x\|_{\ell_{p}} + \|x_{n_j}-x\|_{\ell_{p}}\to 0
\]
which contradicts  the first inequality in (\ref{equ:compact-ell-p-0}). Therefore there is $\lambda>1$ such that
\begin{equation} \label{equ:compact-ell-p-p}
\tfrac{1}{\lambda} \leq b_n=\|x_{n}-x\|_{\ell_{p}}\leq \lambda \quad \mbox{ for every $n\in\w$}.
\end{equation}
Then (\ref{equ:compact-ell-p-p}) implies that the normalized sequence $\{\tfrac{1}{b_n}(x_n-x)\}_{n\in\w}$ is weakly null.

Now, by Proposition 2.1.3 of \cite{Al-Kal} (see also Theorem 1.3.2 of \cite{Al-Kal}), %the Bessaga--Pe{\l}czy\'{n}ski selection principle (we use it in see Corollary 4.27 of \cite{fabian-10}),
there are a basic subsequence $\{x_{n_k}-x\}_{k\in\w}\subseteq \ell_{p}$ of $\{x_n-x\}_{n\in\w}$ and a linear topological isomorphism $R:\overline{\{x_{n_k}-x\}_{k\in\w}}^{\,\ell_p}\to \ell_{p}$ such that
\[
R(x_{n_k}-x)= a_k e_k \quad \mbox{ for every $k\in\w$, where }\; a_k :=b_{n_k}=\|x_{n_k}-x\|_{\ell_{p}},
\]
and such that the subspace $\overline{\{x_{n_k}-x\}_{k\in\w}}^{\,\ell_p}$ is complemented in $\ell_{p}$.

By (\ref{equ:compact-ell-p-p}),  the map $Q:\ell_{p}\to\ell_{p}$ defined by $Q(\xi_k):= (a_k \xi_k)$, $(\xi_k)\in\ell_p$, is a topological linear isomorphism. Hence, replacing $R$ by $Q^{-1}\circ R$ and $K$ by $2K$ (so that, by the absolute convexity of $K$, $x_n-x\in 2K$ for every $n\in\w$), without loss of generality we can assume that $x=0$ and $a_k=1$ for every $k\in\w$. Therefore we obtained that there are a weakly null sequence $\{y_k=x_{n_k}\}_{k\in\w}$ in $K$ and a topological linear isomorphism $R:\overline{\{y_k\}_{k\in\w}}^{\,\ell_p}\to \ell_{p}$ such that
\begin{equation} \label{equ:Lp-weak-Glicl-1}
R(y_k)= e_k \quad \mbox{ for every $k\in\w$},
\end{equation}
and such that the subspace $H:=\overline{\{y_k\}_{k\in\w}}^{\,\ell_p}$ is complemented in $\ell_{p}$. Let $L$ be a topological complement of $H$, so $\ell_p=H\oplus L$.

Set $C:=\cacx\big(\{y_k\}_{k\in\w}\big)$. Then, by the absolute convexity of $K$,  $C\subseteq K$ is a weakly compact and absolutely convex subset of $\ell_p^0$, and hence $C$ is a weakly compact absolutely convex subset of $\ell_p$.
Let $i_p:\ell_1\to\ell_p$ be the natural inclusion. Then (\ref{equ:Lp-weak-Glicl-1}) and the weak compactness and absolute convexity of $C$ imply $R(C)=i_p(B_{\ell_1})\subseteq B_{\ell_p}$. Therefore (recall that $C\subseteq K \subseteq \ell_p^0 =\bigcup_{n\in\w}\IF^n$)
\[
B_{\ell_1}=\bigcup_{n\in\w} i_p^{-1}\big( R(C\cap \IF^n)\big).
\]
Taking into account the Baire property of $B_{\ell_1}$, it follows that $i_p^{-1}\big( R(C\cap \IF^n)\big)$ contains an open subset of $B_{\ell_1}$ for some $n\in\w$. But this is impossible because the injectivity of $i_p$ implies that $i_p^{-1}\big( R(C\cap \IF^n)\big)\subseteq i_p^{-1}\big( R(H\cap \IF^n)\big)$ is finite-dimensional.
This contradiction shows that  $K$ is a compact subset of $\ell_p^0$. Thus $\ell_p^0$ has the weak Glicksberg property.

Since $\ell_p^0$ has the weak Glicksberg property, it has the $DP$ property by (i) of Corollary \ref{c:weak-Glick-DP}.\qed
\end{proof}

\begin{remark} \label{rem:DP-hereditary-1} {\em
(i) For every $q>1$, (ii) of Proposition \ref{p:Lp-Schur}  and (iv) of Proposition \ref{p:Lp-DP}   show that normed spaces with the weak Glicksberg property may not have the $q$-Schur property.

(ii) Proposition \ref{p:Lp-DP} also shows that for all $1\leq p<q\leq\infty$ and each $p',q'\in[1,\infty]$, there are normed spaces which have the $p$-Schur property, the strict $DP_p$ property and the sequential  $DP_{(p,p')}$ property, but which have neither the $q$-Schur property, the strict $DP_q$ property nor the sequential  $DP_{(q,q')}$ property.

(iii) Since $\ell_p^0$ is a large subspace of $\ell_p$ with the $DP$ property and $\ell_p$ does not have the $DP$ property, it follows that the converse to Proposition \ref{p:DP-hereditary-large} is not true in general.\qed}
\end{remark}

The next proposition stands some relationships between the introduced Dunford--Pettis type properties for the important case of locally complete locally convex spaces.
\begin{proposition} \label{p:DP=>seqDP}
Let $p,q\in[1,\infty]$, and let $(E,\tau)$ be a locally complete space.
\begin{enumerate}
\item[{\rm(i)}] If $E$ has the $DP$ property, then $E$ has the strict $DP_p$ property.
\item[{\rm(ii)}] If additionally $E$ is weakly sequentially $p$-angelic, then $E$ has the $DP$ property if and only if $E$ has the strict $DP_p$ property.
\item[{\rm(iii)}]  Assume that $E$ is a $q$-quasibarrelled  space such that $E'_\beta$ is locally complete. If $E$ has the $DP$ property, then it has  the sequential $DP_{(p,q)}$ property.
\item[{\rm(iv)}]  Assume additionally that $E$ is weakly sequentially $p$-angelic and $E'_\beta$ is a weakly sequentially $q$-angelic space. If  $E$ has  the sequential $DP_{(p,q)}$ property, then $E$ has the $DP$ property and  the strict $DP_p$ property.
\end{enumerate}
\end{proposition}

\begin{proof}
(i) By Theorem \ref{t:strict-DPp}, to show that the space $E$ has  the strict $DP_p$ property it suffices to prove that the space $(E,\tau_{\Sigma'})$ has the $p$-Schur property. Let $S=\{x_n\}_{n\in\w}$ be a weakly $p$-summable sequence in $(E,\tau_{\Sigma'})$. Since, by Lemma  \ref{l:Sigma'}, the topology  $\tau_{\Sigma'}$ is compatible with the original topology $\tau$, the sequence $S$ is weakly $p$-summable in $E$. As $E$ is locally complete, the set $K:=\cacx\big(S\big)$ is a weakly compact subset of $E$. Therefore, the $DP$ property of $E$ and  (vii) of Theorem \ref{t:def-DP} imply that $K$ is  compact in the topology $\tau_{\Sigma'}$.
Since $S$ is weakly compact and $\sigma(E,E')\subseteq\tau_{\Sigma'}$, $S$ is a  $\tau_{\Sigma'}$-closed subset of $K$ and hence $S$ is $\tau_{\Sigma'}$-compact. Therefore, by Lemma \ref{l:null-seq}, $S$
is a  $\tau_{\Sigma'}$-null sequence. Thus  the space $(E,\tau_{\Sigma'})$ has the $p$-Schur property.
\smallskip

%\ref{p:Lp-Schur}
(ii) Taking into account (i) we have to show only that if $E$ has the strict $DP_p$ property, then $E$ has the $DP$ property. Let $K$ be an absolutely convex weakly compact subset of $E$. Fix an operator $T$ from $E$ into a Banach space $L$ which transforms bounded sets into relatively weakly compact sets. By (v) of Theorem \ref{t:def-DP},  it suffices to show that $T(K)$ is a compact subset of $L$.  Since $L$ is a Banach space, it suffices to prove that every sequence $\{T(x_n)\}_{n\in\w}$ in $T(K)$ has a subsequence converging to a point of $T(K)$. As $E$ is weakly sequentially $p$-angelic, $K$ is a weakly sequentially $p$-compact subset of $E$, and hence the sequence $\{x_n\}_{n\in\w}$ has a subsequence $\{x_{n_k}\}_{k\in\w}$ which weakly $p$-converges to some point $x\in K$. Since $E$ has the strict $DP_p$ property, $T$ is $p$-convergent and hence $T(x_{n_k})\to T(x)\in T(K)$. Thus $T(K)$ is a compact subset of $L$, and hence $E$ has the $DP$ property.
\smallskip

(iii) By Theorem \ref{t:seq-DPpq}, to show that the space $E$ has  the  sequential $DP_{(p,q)}$ property  it suffices to prove that (1) the topology $\tau_{S_q}$ of uniform convergence on weakly $q$-summable sequences of $E'_\beta$is compatible with the original topology $\tau$, and (2)  the space $(E, \tau_{S_q})$ has the $p$-Schur property.
By (v) of Lemma \ref{l:Sigma-seq}, we have $\sigma(E,E')\subseteq \tau_{S_q} \subseteq \tau_{\Sigma'}$. Now (i) of Lemma  \ref{l:Sigma'} implies that $\tau_{S_q}$ is compatible with $\tau$ and hence (1) is satisfied. As we showed in (i), the space $(E,\tau_{\Sigma'})$  has  the $p$-Schur property. Since $\sigma(E,E')\subseteq \tau_{S_q} \subseteq \tau_{\Sigma'}$, it follows that $(E, \tau_{S_q})$ has  the $p$-Schur property,  so (2) holds true.
\smallskip

(iv) Taking into account (ii) it suffices to show that $E$ has  the strict $DP_p$ property. To this end, by Theorem \ref{t:strict-DPp}, it suffices to prove that the space $(E,\tau_{\Sigma'})$ has the $p$-Schur property.

Suppose for a contradiction that $(E,\tau_{\Sigma'})$ does not have the $p$-Schur property. Since $\tau_{\Sigma'}$ is compatible with $\tau$, there exists a weakly $p$-summable sequence $S=\{x_n\}_{n\in\w}$ in $E$ which does not converge to $0$ in $\tau_{\Sigma'}$. Therefore there are $\e>0$, a strictly increasing  sequence $\{n_k\}_{k\in\w}$ in $\w$, and a sequence $\{\chi_k\}_{k\in\w}$  in some $K\in \Sigma'(E')$ such that
\begin{equation} \label{equ:DP-seqDP-1}
\big|\langle \chi_k, x_{n_k}\rangle\big|\geq \e \quad \mbox{ for every $k\in\w$}.
\end{equation}
By the definition of $\Sigma'(E')$, the set  $K$ is weakly compact in $E'_\beta$. Therefore, by the  weak sequential $q$-angelicity of $E'_\beta$, there is a subsequence $\{\chi_{k_j}\}_{j\in\w}$ of $\{\chi_k\}_{k\in\w}$ which weakly $q$-converges to some $\chi\in K$. Since $S$ is weakly $p$-summable so is its subsequence $\{ x_{{n_{k_j}}}\}_{j\in\w}$. Then (\ref{equ:DP-seqDP-1}) and the sequential $DP_{(p,q)}$ property of $E$ imply
\[
\e\leq \big|\langle \chi_{k_j}, x_{{n_{k_j}}}\rangle\big|\leq  \big|\langle \chi_{k_j}-\chi, x_{{n_{k_j}}}\rangle\big| + \big|\langle \chi, x_{{n_{k_j}}}\rangle\big|\to 0 \quad \mbox{ as $j\to\infty$}.
\]
This contradiction shows that $(E,\tau_{\Sigma'})$ has the $p$-Schur property. \qed
\end{proof}
%By Proposition 3.3(ii) of \cite{ABR}, if both $E$ and $E'_\beta$ are weakly sequentially angelic and $E$ is quasi-complete, then
Proposition 3.3(i) of \cite{ABR} states that if $E$ is a barrelled quasi-complete lcs with the $DP$ property, then $E$ has the sequential $DP$ property. The next corollary essentially generalizes this result.

\begin{corollary} \label{c:DP=>sDPp}
Let $E$ be a quasibarrelled locally complete space. If $E$ has the $DP$ property, then $E$ has  the strict $DP_p$ property and the sequential $DP_{(p,q)}$ property for all  $p,q\in[1,\infty]$.
\end{corollary}

\begin{proof}
It is clear that $E$ is $q$-quasibarrelled for all $q\in[1,\infty]$. By Theorem 12.1.4 of \cite{Jar}, the space $E'_\beta$ is locally complete. Now (i) and (iii) of Proposition \ref{p:DP=>seqDP} apply.\qed
\end{proof}

\begin{corollary} \label{c:DP<=>sDPp}
Let $p,q\in[1,\infty]$, and let $E$ be a quasibarrelled, locally complete space such that $E$ is weakly sequentially $p$-angelic and $E'_\beta$ is a weakly sequentially $q$-angelic space. Then the following assertions are equivalent:
\begin{enumerate}
\item[{\rm(i)}] $E$ has the $DP$ property;
\item[{\rm(ii)}]  $E$ has  the strict $DP_p$ property;
\item[{\rm(iii)}] $E$ has the sequential $DP_{(p,q)}$ property.
\end{enumerate}
\end{corollary}

\begin{proof}
The equivalence (i)$\Leftrightarrow$(ii) follows from (ii) of Proposition \ref{p:DP=>seqDP}. The implication (i)$\Ra$(iii) follows from Corollary \ref{c:DP=>sDPp}, and (iii)$\Ra$(i) follows from (iv) of Proposition \ref{p:DP=>seqDP}.\qed
\end{proof}

If $p=q=\infty$, the Eberlein--\v{S}mulian theorem and Corollary \ref{c:DP<=>sDPp} immediately imply the following classical Grothendieck result.
\begin{corollary} \label{c:DP=sDP}
A Banach space $E$ has the $DP$ property if and only if it has the sequential $DP$ property.
\end{corollary}

%\begin{proof}
%(i)$\Ra$(ii') Let $T:E\to L$  be an operator from $E$ into a Banach space $L$ which  transforms bounded sets into relatively weakly compact sets. Since $E$ has the strict $DP_p$ property, $T$ is $p$-convergent. Let $A$ be a weakly sequentially $p$-precompact subset of $E$. Then, by Proposition \ref{p:p-convergent-s}, $T(A)$ is  relatively sequentially compact in $L$. As $L$ is a Banach space, it follows that $T(A)$ is relatively compact in $L$, as desired.

%(ii')$\Ra$(iii') is trivial.

%(iii')$\Ra$(i) Let $T:E\to L$  be an operator from $E$ into a Banach space $L$ which  transforms bounded sets into relatively weakly compact sets. Fix a weakly $p$-summable sequence $S=\{x_n\}_{n\in\w}$ in $E$. Since $S$ is weakly sequentially $p$-compact, by (iii), the set $T(S)=\{T(x_n)\}_{n\in\w}$ is relatively compact in $L$. As $T$ is weakly continuous and $S$ is a weakly null-sequence, we obtain that $T(x_n)\to 0$ in the weak topology of $L$. Therefore $0$ is a unique weak limit point of $T(S)$. Since $L$ is metrizable and $T(S)$ is relatively compact, it follows that $T(x_n)\to 0$ in the norm topology of $L$. Thus $T$ is $p$-convergent.\qed
%\end{proof}

%%%%%%%%%%%%%%%%%%%%%%%%%%%%%%%%%%%%%
%%%%%%%%%%%%%%%%%%%%%%%%%%%%%%%%%%%%%
%%%%%%%%%%%%%%%%%%%%%%%%%%%%%%%%%%%%%
%%%%%%%%%%%%%%%%%%%%%%%%%%%%%%%%%%%%%
%%%%%%%%%%%%%%%%%%%%%%%%%%%%%%%%%%%%%

\bibliographystyle{amsplain}

\end{document}